\documentclass[10pt,reqno]{amsart}

 \usepackage{amsmath,amsfonts,amssymb,amscd,amsthm,amsbsy,bbm, epsf,calc,enumerate,color,graphicx,verbatim,cite,import}
\usepackage{color}
\usepackage{datetime,appendix}
\usepackage{chapterbib,epstopdf}

\usepackage{float}

\usepackage{ifpdf}
\ifpdf
    \usepackage[colorlinks,citecolor=black,linkcolor=black]{hyperref}
\fi

\newcommand{\lb}{\label}

\theoremstyle{plain}
\newtheorem{definition}{Definition}[section]
\newtheorem{thm}{Theorem}[section]
\newtheorem{theorem}{Theorem}[section]
\newtheorem{mainthm}{Theorem}

\newtheorem{lemma}[thm]{Lemma}

\newtheorem{corollary}[thm]{Corollary}

\newtheorem{proposition}[thm]{Proposition}
\theoremstyle{definition}

\newtheorem{example}[thm]{Example}

\theoremstyle{remark}                  %% For unnumbered Remarks, etc.

\newtheorem{remark}[thm]{Remark}

\DeclareMathOperator*{\osc}{osc}

\definecolor{darkgreen}{rgb}{0,0.4,0}

\newcommand{\beq}{\begin{equation}}
\newcommand{\eeq}{\end{equation}}

\newcommand{\bbN}{{\mathbb{N}}}
\newcommand{\bbR}{{\mathbb{R}}}

\newcommand{\bbZ}{{\mathbb{Z}}}

\newcommand{\calM}{{\mathcal M}}

\newcommand{\eps}{\varepsilon}

\numberwithin{equation}{section}

\def\XXint#1#2#3{{\setbox0=\hbox{$#1{#2#3}{\int}$ }
\vcenter{\hbox{$#2#3$ }}\kern-.6\wd0}}

\title[Homogenization Problems on Space-Time Domains]{On Homogenization Problems with Oscillating Dirichlet Conditions in Space-Time Domains}

\author[Y Zhang]{\bfseries Yuming Paul Zhang}

\address{
Department of Mathematics \\ % \hfill (Received 00 00 2010)\\
University of California,   San Diego\\ %\hfill (Revised  00 00 2010)\\
9500 Gilman Dr\\
La Jolla, CA 92093}
\email{yuz018@ucsd.edu}

\begin{document}

\vspace{18mm} \setcounter{page}{1} \thispagestyle{empty}

\begin{abstract}
We prove the homogenization of fully nonlinear parabolic equations with periodic oscillating Dirichlet boundary conditions on certain general prescribed space-time domains. It was proved in  \cite{wf,Kimcont} that for elliptic equations, the homogenized boundary data exists at boundary points with irrational normal directions, and it is generically discontinuous elsewhere. However for parabolic problems, on a flat moving part of the boundary, we prove the existence of continuous homogenized boundary data $\bar{g}$. 
We also show that, unlike the elliptic case, $\bar{g}$ can be discontinuous even if the operator is rotation/reflection invariant.
%Assuming that $\bar{g}$ is only discontinuous on a small subset of the lateral boundary, we prove the homogenization result. %The proof relies on the observation of two boundary layers of different scales. With the homogenized boundary data that is possibly discontinuous in a small subset, we prove the homogenization of the equation. 
\end{abstract}

\maketitle

\vspace{.1cm}

\vspace{.1cm}

\bigskip

\section{Introduction}
We investigate the homogenization problem of fully nonlinear parabolic equations in general space-time domains. Fix $T>0$ and $d\in \bbN^+$, let $\Omega_T$ be an open bounded subset of $\mathbb{R}^d\times (0,T)$ of the following form
\[
 \Omega_T:=\bigcup_{0<t<T}\left(\Omega(t)\times\{t\}\right)\]
for some $\Omega(t)\subset \mathbb{R}^d$.
Write $\partial_l\Omega_T:=\bigcup_{0<t<T}
\left(\partial\Omega(t)\times\{t\}\right)$ as the lateral boundary of $\Omega_T$. We consider the following problem: 
\begin{equation}\label{eqns1} 
\left\{\begin{aligned}
&\frac{\partial}{\partial t}u^\eps(x,t)-F(D^2u^\eps,x,\frac{x}{\eps},t,\frac{t}{\eps^2})=0 &\text{ in }&\quad \Omega_T,\\
&u^\eps(x,t)=g(x,\frac{x}{\eps},t,\frac{t}{\eps^2}) &\text{ on }&\quad\partial_p\Omega_T:=\partial_l\Omega_T\cup {( \Omega(0)\times\{0\})},
\end{aligned}
\right.
\end{equation}
where $F(M,x,y,t,s)$ is uniformly elliptic in $M$, and both $F$ and $g(x,y,t,s)$ are $\mathbb{Z}^d/\mathbb{Z}$-periodic in $y/s$ variables. Section \ref{assumption} will give the precise assumptions {we make} on the operator, boundary data and the domain. When $\eps>0$ is small, both the operator and the boundary data involve large oscillations. The goal in this paper is to understand the averaging behaviour of solutions $u^\eps(x,t)$ as $\eps\rightarrow 0$, especially when $(x,t)$ is close to the boundary.
If $u^\eps$ converges, we say the problem \eqref{eqns1} homogenizes.
%More details can be found in Theorem \ref{thmA}.

\smallskip

%Homogenization problems have a long history.

Homogenization problems have a long history and we will only mention a small portion of works, particularly nonlinear periodic problems, that are closely related to our work. We refer readers to \cite{20} and the book \cite{bensoussan2011asymptotic} for the extensive bibliography. For problems with non-oscillating boundary data,
Garc\'ia-Azorero et al. \cite{garcia2003homogenization} proved homogenization for a class of quasilinear parabolic problem in divergence form. Later the fully nonlinear, non-divergence form problem was studied by Marchi \cite{inthomo}.
For general uniformly elliptic {equations} but with oscillating Dirichlet boundary data, Barles and Mironescu \cite{4} worked on half-planes with boundaries passing through the origin. In their paper, the homogenized boundary data arises as a boundary layer limit of a problem set in half {spaces}. 

In general domains, the homogenized boundary data cannot be identified at boundary points of rational normal (the normal vector lies in $\mathbb{R}\mathbb{Z}^d$).
{Despite possible discontinuities at rational directions, Feldman \cite{wf} proved that homogenization happens when there is no flat portions on the boundary of the domain (i.e. the set of boundary points of rational normal has a small Hausdorff dimension). Feldman and Kim showed in \cite{Kimcont} that the homogenized boundary data is H\"{o}lder continuous if the homogenized operator is either rotation/reflection invariant or linear.} More recently, continuity of the homogenized boundary data for linear elliptic system of divergence form is proved in \cite{shen2017regularity,feldman2018continuity}.

%Also in \cite{kim}, Choi and Kim worked on elliptic operators but with oscillating Neumann condition. They showed that: if the domain does not have flat boundary parts and the homogenized operator is rotation/reflection invariant, then solutions uniformly converge and the homogenized boundary data can be continuously defined on the entire boundary.

\smallskip

In this work, under some {conditions} on the domain, we are going to show that the solutions $u^\eps$ of \eqref{eqns1} converge locally uniformly to $\bar{u}$ which solves a parabolic equation:
\begin{equation}\label{homosolution}
\left\{\begin{aligned}&
\frac{\partial}{\partial t}\bar{u}(x,t)-\bar{F}(D^2 \bar{u},x,t)=0 &\text{ in }&\quad\Omega_T,\\
&\bar{u}(x,t)=\bar{g}(x,t) &\text{ on }&\quad\partial_p\Omega_T.
\end{aligned}
\right.
\end{equation}
Here $\bar{F}$ is the homogenized operator from \cite{intho,inthomo} which will be recalled in Section 2.2.2. And $\bar{g}$ is called the homogenized boundary data which is supposed to be the limit, if exists, of $u^\eps(x,t)$ as $(x,t)$ approaches the boundary. %If such a limit exists at $(x_0,t_0)\in \partial_p\Omega_T$, we say that $u^\eps$ homogenizes at $(x_0,t_0)$. 

To study the limit of $u^\eps$ near the boundary, we need to introduce the \textit{boundary layer problem} which can also be called the \textit{cell problem}, also see \cite{4,wf,feldman2018continuity}. Let us take a lateral boundary point $(x_0,t_0)\in\partial_l\Omega_T$ and denote its interior spatial normal as $\nu_0=\nu_{x_0,t_0}$. As done in the literatures, the analysis proceeds by blowing up a sequence of functions 
\[
{v^\eps(x,t):=u^\eps(x_0+\eps x,t_0+\eps^2 t).}
\]
By passing to a subsequence of $\eps\to 0$, suppose $\frac{x_0}{\eps}\to z$ in $\mathbb{R}^d/\mathbb{Z}^d$ and $\frac{t_0}{\eps}\to \tau$ in $\mathbb{R}/\mathbb{Z}$. Then formally taking $\eps\to 0$
 in \eqref{eqns1} gives rise to the following boundary layer problem:
\begin{equation*}
{\text{(I)}}\quad
\left\{\begin{aligned}
&\frac{\partial v^{z,\tau}}{\partial t}-{F}(D^2 v^{z,\tau},x_0,{x}+z,t_0,t+\tau)=0 &\text{ in }&\quad P_{\nu_0}:=\{(x,t)\in\mathbb{R}^{d+1}\,|\,x\cdot\nu_0\geq 0\},\\
&v^{z,\tau}(x,t)=g(x_0,x+z,t_0,t+\tau) &\text{ on }&\quad\partial P_{\nu_0}.
\end{aligned}
\right.
\end{equation*}
The solution to (I) has a \textit{boundary layer limit}:
\begin{equation}
    \label{bd layer limit}
    \varphi^{z,\tau}_{\nu_0}:=\lim_{R\to\infty}v^{z,\tau}(x+R\nu_0,t),
\end{equation}
and the limit can be shown to be independent of $\tau,x,t$. It is typically not independent of the translation $z$.
When $\nu_0$ is irrational ($\nu_0\in S^{d-1}\backslash\mathbb{R}\mathbb{Z}^d$), similarly as in the elliptic problem \cite{wf}, $\varphi^{z,\tau}_{\nu_0}$ is also independent of $z$. {For these $(x_0,t_0)$}, we can identify
$\bar{g}(x_0,t_0)$ to be the boundary layer limit $\varphi^{z,\tau}_{\nu_0}$. In fact it is not hard to show
\[\lim_{R\to\infty}\limsup_{\eps\to 0}|u^{\eps}(x_\eps+\eps R\nu_0,t_\eps)-\varphi_{\nu_0}^{z,\tau}|=0\]
for all $\Omega_T\ni(x_\eps,t_\eps)\to(x_0,t_0)$ as $\eps\to 0$.
However when the normal direction $\nu_0$ is rational, there is a serious problem that the limit \eqref{bd layer limit} depends on $z$. Actually it was proved in \cite{Kimcont} in the elliptic setting that, generically, the limit of $u^{\eps}$ cannot be continuous on the boundary. To resolve this problem, as mentioned before, \cite{wf,Kimcont} assumed that boundary points with rational direction is a small set, and so they need the dimension $d\geq 2$. 

{
The novelty of this paper is that we provide another solution to the above problem (and we allow $d=1$). We observe that for the parabolic equation, when the lateral boundary is moving along time, boundary homogenization occurs even on spatially flat part with rational normal. }
In order to see the ideas behind, let us go back to the boundary layer problem (I) with rational $\nu_0$ and explain the reason why the boundary layer limit depends on $z$.
For $x\in \partial P_{\nu_0}$ and $z_1,z_2$ such that $z_1\cdot \nu_0\neq z_2\cdot \nu_0$ in $\mathbb{R}^d/\mathbb{Z}^d$, then $\partial P_{\nu_0}+z_1$ and $\partial P_{\nu_0}+z_2$ are different hyperplanes (even after shifting in directions that are perpendicular to $\nu_0$), and the values of $g(x_0,\cdot,t_0,t+\tau)$ on them can be very different. Therefore the corresponding boundary layer limit depends on $z$. However for space-time domain, if the boundary is moving, {we expect that the solution will see all values of $g(x_0,\cdot,t_0,\cdot)$ in a larger scale.} Using this approach, we can also show homogenization for $d=1$. %Indeed boundary homogenization occurs in this situation through a double boundary layer limits procedure.

%\subsection{the double scale homogenization. }

\smallskip

\begin{figure}\label{fig double}\caption{Double-scale boundary layer limits}\centering\includegraphics[scale=0.5]{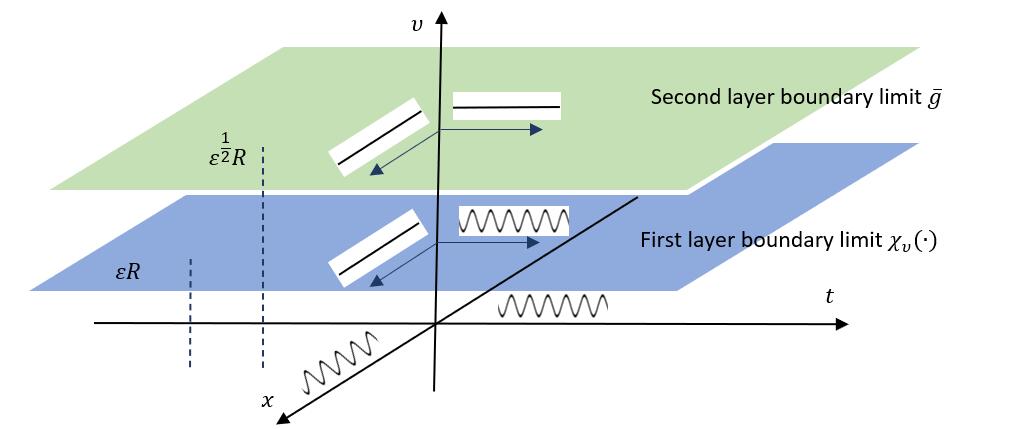}\end{figure}

Now we proceed to describe more precisely the mechanism of the homogenization at flat moving boundary with a rational direction $\nu$.
Fix one such boundary point and for simplicity we assume it is $(x_0,0)$ (there is no essential difficulties of considering $(x_0,t_0)$). Then write $c_0\neq 0$ as the boundary speed (in inner normal direction) at $(x_0,0)$. In order to see how the moving boundary could help, consider the following transformation
\begin{equation*}
\mu^{\eps}(x,t){:=} u^\eps(x_\eps+\eps x+c_0(t_\eps+\eps^2 t)\,\nu,t_\eps+\eps^2  t)
\end{equation*}
where $x_\eps,t_\eps$ are going to be specified.
%This localization takes the movement of the domain and two different scales into consideration. Here $(y,s)\in\mathbb{R}^{d+1}$ are used as space-time parameters in a larger scale ($\eps^\frac{1}{2}$-scale in space and $\eps$-scale in time).
By passing $\eps\rightarrow 0$ along subsequences, we assume that 
\beq\lb{1.2}
\frac{x_\eps+c_0(t_\eps+\eps^2 t)\nu}{\eps}\to
z_*\quad\text{ modulus }\mathbb{Z}^d.
\eeq
Then, similarly as done in \eqref{bd layer limit}, we obtain a boundary layer limit $
\chi_{\nu}^{z_*}$ from a boundary layer problem (R1) (see section 3.2) (since $t_0=0$, there is no $\tau$ dependence here). But now $\chi_{\nu}^{z_*}$ depends on $z_*$, and we can show that it only depends on $z_*\cdot\nu$. Then this $\chi_{\nu}^{z_*}$ provides the value of $u^\eps$ at $(x_\eps+R\eps\nu+c_0t_\eps\nu,t_\eps)$ for large $R>1$ with a small error depending on $R^{-1},\eps$ (along the subsequence of $\eps\to 0$).

Now we pull ourselves further away from the boundary on which $(x_0,0)$ resides to a distance of $R\eps^{\frac{1}{2}}$ (from a distance of $R\eps$) by sending
\[
x_\eps:=x_0+\eps^\frac{1}{2}y,\quad t_\eps:=\eps s.
\]
Then assuming $\frac{x_0}{\eps}\to z$ in $\bbR^d/\bbZ^d$ along subsequence of $\eps\to 0$, \eqref{1.2} yields $z_*\cdot\nu=z\cdot\nu+c_0s$ for all $y\in \{y'\,|\,y'\cdot\nu=0\}$.
Next define
\[
{w^\eps(y,s):= u^\eps(x_0+\eps^\frac{1}{2} y+c_0\eps s\,\nu,\eps s).}
\]
which gives rise to (after ignoring higher order terms)
\[\frac{\partial}{\partial t}w^\eps(y,s)- {F}(D^2_yw^\eps,x_0,\eps^{-\frac{1}{2}}y+z+c_0s\nu,t_0,{\eps^{-1}}s )\approx 0  \quad \text{ in } \quad P_{\nu}.\]
As for the boundary information, we will use $\chi_\nu(z\cdot\nu+c_0s):=\chi_\nu^{z\cdot\nu+c_0s}$ which is approximately the value of $u^\eps$ near the boundary but of distance of the small scale $R\eps$ in space.
This equation is a homogenization problem with non-oscillating boundary data. The classical theorem suggests that its limit is
\begin{equation*}
\text{(R2)}\quad\left\{
\begin{aligned}&
{\partial_t}w^z(y,s)- \bar{F}^0(D^2w )=0 &  \text{ in }&\quad P_{\nu},\\
&w^z_2(y,s)=\chi_{\nu}(z\cdot\nu+c_0s) &  \text{ on }&\quad \partial P_{\nu}
\end{aligned}
\right.
\end{equation*}
where $\bar{F}^0(M)$ is the homogenized operator of $F(M,x_0,\cdot,t_0,\cdot)$. This is our second boundary layer problem. Though the boundary data in (R2) depends on $z$, the boundary layer limit is independent of $z$ (the choice of subsequences of $\eps\rightarrow 0$) due to the fact that different $z$'s only cause {some shifts} in the time direction.
From the above informal argument, we expect that at a larger scale ($\eps^\frac{1}{2}$ in space and $\eps$ in time), the solutions converge at the boundary:
\[
\lim_{\eps\to 0}u^{\eps}(x_0+\eps^{\frac{1}{2}} R\nu,0)=\lim_{\eps\to 0}w^{\eps}(R\nu,0)=w^z(R\nu,0)
\]
Thus we define the boundary layer limit of $w^z$ as the homogenized boundary data $\bar{g}$ at $(x_0,0)$. 

At the end of the story, we will show boundary homogenization in the following subsets of the lateral boundary:
\begin{equation}\label{bdry clss}
    \begin{aligned}
 \Gamma_1(T):=&\left\{ (x,t)\,|\,  {t\in (0,T),\,x\in \partial\Omega(t) }, \, \nu_{x,t} \text{ is irrational}\right\},\\
 \Gamma_2(T):=&\left\{ (x,t)\,|\, t\in (0,T),\,x\in \partial\Omega(t), \, \nu_{y,s}=\nu \text{ is rational and constant},\right.
\\
&\quad  {\text{  for all $(y,s)\in \partial_l\Omega$ that are sufficiently close to }}(x,t) \}.\end{aligned}
\end{equation}
Later we drop $(T)$ from $\Gamma_1(T)$ and $\Gamma_2(T)$ for simplicity.
Let us remark that when $d=1$, then the unit normal $\nu_{x,t}$ is always to a rational direction, and therefore $\Gamma_1=\emptyset$. %In contrast, \cite{wf,Kimcont} requires the spatial dimension $d\geq 2$.

Next we discuss the continuity property of $\bar{g}$. It is not hard to show that $\bar{g}$ is continuous on $\Gamma_T:=\Gamma_1\cup \Gamma_2\in \partial_l\Omega_T$. However we will show by examples that $\bar{g}$ cannot be continuous extend to $\overline{\Gamma}_T$ (where the closure is taken in $\partial_l\Omega_T$). Let us mention that in \cite{Kimcont, kim} it was proved for the elliptic problems, if the operators homogenize to a rotation/reflection invariant or a linear operator, then the homogenized boundary data is continuous on the closure of the set where it is well-defined. Nevertheless, in our setting, if $\bar{F}$ is rotation/reflection invariant, $\bar{g}$ can be continuously defined on $\overline{\Gamma}_1\backslash \overline{\Gamma}_2$ (but still not $\overline{\Gamma}_T$). We show the possibility of the discontinuities in Example \ref{prop4.6}. Finally in the case when $\bar{F}$ is linear, $\bar{g}$ can be continuously extended to $\overline{\Gamma}_T$.

%$\bar{g}$ is generically discontinuous on $\partial_l \Omega_T$ which is proved in \cite{Kimcont} for elliptic problems.

%\begin{figure}\caption{Double-scale boundary layer limits}\centering\includegraphics[scale=0.4]{homo4.JPG}\end{figure}

After identifying $\bar{g}$ on {$\Gamma_T$} as well as the bottom boundary $\Omega(0)\times\{0\}$, and showing the continuity properties of $\bar{g}$, we are in the position of proving the homogenization result despite a possible lack of continuity in the homogenized boundary data. This will be done by a comparison principle which is used by Feldman \cite{wf} in the proof of homogenization for the elliptic problem.
%Given the possible discontinuity of the boundary data, we will show the well-posedness of \eqref{homosolution} (see Theorem \ref{uniq}) through a comparison principle and Perron's method under the condition that the collection of undefined points on the boundary is of small Hausdorff dimension, depending on the ellipticity constants of the operator. For this purpose we will construct singular super solutions to the parabolic equations with Pucci operators. 
%The similar well-posedness theorem for the fully nonlinear elliptic equation is given in Theorem 4.2 \cite{wf}.

\smallskip

Now we state our main theorem:

\begin{mainthm}\label{thmA}(Theorem \ref{final})
Let $d\geq 1$, and assume conditions (F1)--(F4)(O) hold (see Section \ref{assumption} for details). Suppose the Hausdorff dimension of $\partial_l\Omega_T\backslash \Gamma_T$ is less than $\frac{d\lambda}{2\Lambda}$ where $\lambda,\Lambda$ are the elliptic constants of the operator. Then \eqref{eqns1} homogenizes in the sense that $u^\eps$ converges locally uniformly to the unique solution of the homogenized problem
\begin{equation}\lb{a11}
\left\{\begin{aligned}&
\frac{\partial}{\partial t}\bar{u}(x,t)-\bar{F}(D^2 \bar{u},x,t)=0 &\text{ in }&\quad\Omega_T,\\
&\bar{u}(x,t)=\bar{g}(x,t) &\text{ on }&\quad\Gamma_T\cup (\Omega(0)\times\{0\}),
\end{aligned}
\right.
\end{equation}
{where $\bar{F}$ and $\bar{g}$ are, respectively, the homogenized operator and the homogenized boundary data.}
\end{mainthm}
We remark that a slightly weaker hypothesis on the size of $\partial_l\Omega_T\backslash \Gamma_T$ is required if we use of the parabolic Hausdorff dimension (see Definition \ref{prblc HD}). However, to the best of our knowledge, the sharp hypothesis that guarantees the well-posedness of \eqref{a11} remains open.%general parabolic with discontinuous boundary data remains open.

\smallskip

Concerning linear operators, we only need to assume that there is no flat stationary lateral boundary.

\begin{mainthm}(Example \ref{prop4.6}, Theorem \ref{linearversion})
Let $d\geq 1$, and assume conditions (F1)--(F4)(O) hold. Let $\Omega_T$ be a time-dependent domain such that $\partial_l\Omega_T\subset \overline{\Gamma}_T$. Then if $\bar{F}$ is a linear operator (particularly when $F$ is linear), problem \eqref{eqns1} homogenizes with continuous homogenized boundary data $\bar{g}$ on $\partial_p\Omega_T$. If $\bar{F}$ is not linear, $\bar{g}$ can be discontinuous on $\partial_l\Omega_T$.
\end{mainthm}

%$\circ$ {\it Sketch of the Proof for Theorem \ref{thmA}}
%\medskip

\subsection{Outline}

In Section 2 we collect most important notation, the assumptions, and some previous results. We will prove one localization lemma and, as a corollary, a comparison principle in $P_\nu$.
The boundary layer problems for the lateral boundary with both irrational and rational normal will be formulated in Section 3. From these problems, we define the homogenized boundary data $\bar{g}$.
In Section 4 we study the continuity properties of $\bar{g}$. In \ref{discont}, we give two examples showing the failure of continuous extension of $\bar{g}$ on $\overline{\Gamma}_1\cap\overline{\Gamma}_2$.
In Section 5, we show the local uniform convergences of $u^\eps$ to {$\bar{g}$ on $\Gamma_T$}. Section 6 discusses the homogenization on the bottom boundary. In Section 7, we prove the well-posedness of the homogenized problem given discontinuous boundary data, and finish the proofs of Theorem A and Theorem B.

\subsection{Acknowledgements}

The author would like to thank his advisor Inwon Kim for suggesting the problem, as well as for the stimulating guidance and discussions. The author would also like to thank William M. Feldman and Olga Turanova for helpful discussions.

\section{Notations, Assumptions and Preliminaries}\label{assumption}

%We write $\mathcal{M}=\{d\times d$ symmetric \textcolor{blue}{matrices} with real entries$\}$ . 
{Let $\mathcal{M}$ be the space of $d\times d$ symmetric matrices equipped with the spectral norm.}

By rational vector or a vector to a rational direction we mean $\nu\in\mathbb{R}\mathbb{Z}^d\cap \mathbb{S}^{d-1}$. And we call a unit vector irrational if it is not a rational vector. In this paper, $\nu$ will always be unit spatial vectors.

We will consider a bounded time interval: $[0,T]$ for some $T>0$, and by passing $T\to +\infty$ the results hold for all time. Let us call $\Omega(0)\times\{0\}$ the bottom boundary of $\Omega_T$ and $\partial_l \Omega_T$ the lateral boundary. The union of the bottom and lateral boundary is denoted by $\partial_p \Omega_T$. 
%Sometimes we omit $T$ if no confusion arises. 

By universal constants, we mean constants which only depend on $d$, $\lambda$, $\Lambda$, bounds on $F,g$ given in condition (F1) and the $\mathcal{C}^1$ norm of $\partial_l\Omega_T$.

%Seeing from the condition (O) below, $c(x,t)$ can be extended to a continuous function in $\mathbb{R}^{d+1}$ with support in a neighborhood of the lateral boundary. We will use one such extension throughout the paper.

By parabolic cylinders of radius $r$ we mean \begin{equation}\label{prblc}
S_r(x,t):=\left\{(y,s)\,|\,|y-x|< r, |s-t|< r^2\right\}.
\end{equation} 
We may omit $x,t$ if they are $0$.

%We use the following notations for some medium limits of functions $f_1(R),f_2(\eps)$:
%\[\widetilde{\lim}_R f_1(R):=\frac{1}{2}\left(\limsup_{R\rightarrow\infty}f_1(R)+\liminf_{R\rightarrow\infty}f_1(R)\right),\]
%\[\widetilde{\lim}_\eps f_2(\eps):=\frac{1}{2}\left(\limsup_{\eps\rightarrow 0}f_2(\eps)+\liminf_{\eps\rightarrow 0}f_2(\eps)\right).\]

Recall the definition of \textit{half-relaxed limits} of a sequence of functions {$u^\eps(x,t):\Omega\to\bbR$}:
\[
\begin{aligned}
u^*(x,t):=\limsup_{\substack{ \Omega\ni(x',t')\rightarrow (x,t),\\\eps\rightarrow 0}}u^\eps(x',t'),\\ u_*(x,t):=\liminf_{\substack{ \Omega\ni(x',t')\rightarrow (x,t),\\\eps\rightarrow 0}}u^\eps(x',t').
\end{aligned}
\]
We write them as $\limsup^* u^\eps(x,t), \liminf_{*} u^\eps(x,t)$ respectively for abbreviation of notation.

The oscillation operator $\osc$ is defined to be:
for any $U\subseteq \mathbb{R}^{d+1}$ and a continuous function $u:U\to\mathbb{R}$,
\[ \osc_{(U)}(u):=\sup \left\{|u(x_2,t_2)-u(x_1,t_1)|\,|\, (x_i,t_i)\in U,\,i=1,2\right\}.\]

\subsection{The assumptions}
Suppose $F$ is a function on $\mathcal{M}\times\mathbb{R}^{2d+2}$ and $g$ is a function on $\mathbb{R}^{2d+2}$. For any $(M_i,x_i,y_i,t,s),(M,x,y,t,s)\in \mathcal{M}\times \mathbb{R}^{2d+2}, i=1,2$, we assume
\begin{itemize}
    \item[(F1) ] $F(M,x,y,t,s)$ is locally Lipschitz continuous in all its parameters and there is a function $\rho\in C[0,\infty)$ with $\rho(0^+)=0$ such that 
\begin{align*}
    |F(M_1,x_1,y_1,t,s)-F(M_2,x_2,y_2,t,s)|&\leq\\ \rho(|M_1-M_2|+(|x_1-x_2|&+|y_1-y_2|)(1+|M_1|+|M_2|)).
\end{align*}
The boundary data $g(x,y,t,s)$ is $\alpha$-H\"{o}lder continuous in all its parameters.

\item[(F2) ] There exist $\Lambda>\lambda>0$ such that for any $ \mathcal{M}\ni N\geq 0$,
\[\lambda {\text{Tr}}\,(N)\leq F(M+N,x,y,t,s)-F(M,x,y,t,s)\leq \Lambda {\text{Tr}}\,(N).\]

\item[(F3) ] $F(M,x,y,t,s), g(x,y,t,s)$ are $\mathbb{Z}^d/\mathbb{Z}$ periodic in $y/s$. 

\item[(F4) ] {For any $(x,t)$, $\eps F(\eps^{-1}M,x,y,t,s)\to F^{x,t}(M,y,s)$ locally uniformly as $\eps\to 0$.
For any $(M,x,t)$, there exists a neighbourhood of $(M,x,t)$ such that for any $(M',x',t')$ inside the neighbourhood, there is a constant $C(M,x,t)$ such that}
\[|F(M'+N,x',y,t',s)- F^{x',t'}(N,y,s)|\leq C \quad \text{ for all } (N, y,s)\in\mathcal{M}\times\mathbb{R}^{d+1}.\]

\end{itemize}

The uniform elliptic condition (F2) induces condition (A1) in \cite{inthomo} and from which we have comparison principle. Condition (F4) will be used for homogenization near the boundary. The second part of (F4) is the same as (A3) in \cite{inthomo}. 
The assumptions made in \cite{4, intho, LW} hold if assuming (F1)--(F4), and so we can apply their results directly.
Operators of the following forms are covered
\[F(M,x,y,t,s)=f(x,y,t,s)+\inf_\beta \sup_\alpha \text{Tr}\,(A_{\alpha\beta}(x,y,t,s)M)  \]
(where $A_{\alpha\beta}(x,y,t,s)\in \calM$ is positive definite, and satisfies some continuity and periodicity conditions).
While our method does not fit in the case when the operator is of divergence form.

\smallskip

For the domain, we assume
\begin{itemize}
    \item[(O) ]  $\Omega(t)$ is open, bounded and connected for every $0\leq t\leq T$. 
The lateral boundary $\partial_l\Omega_{T}$ is $\mathcal{C}^1$ in both space and time.
\end{itemize}
%More conditions will be assumed on the boundary as can be seen in the main Theorem \ref{thmA}.

\smallskip

Next let $G(M,y,s)$ be a function on $\mathcal{M}\times\mathbb{R}^{d+1}$ and $h(y,s)$ a function on $\mathbb{R}^{d+1}$. We say $G,h$ satisfies condition {(G)} if 
\begin{itemize}
    \item[(G1) ] $G(M,y,s)$ is locally Lipschitz continuous in all its parameters and there is a function $\rho\in C[0,\infty)$ with $\rho(0^+)=0$ such that 
\[|G(M_1,y_1,s)-G(M_2,y_2,s)|\leq \rho\left(|M_1-M_2|+(|y_1-y_2|)(1+|M_1|+|M_2|)\right).\]
$h(y,s)$ is $\alpha$-H\"{o}lder continuous in $y,s$.

\item[(G2) ] There exist $\Lambda>\lambda>0$ such that for any $\mathcal{M}\ni N\geq 0$,
\[\lambda {\text{Tr}}\,(N)\leq G(M+N,y,s)-G(M,y,s)\leq \Lambda {\text{Tr}}\,(N).\]

\item[(G3) ] $G(M,y,s), h(y,s)$ are $\mathbb{Z}^d/\mathbb{Z}$ periodic in $y/s$. 

\item[(G4) ] $G$ is homogeneous in $M$ in the sense that for all $c\in\mathbb{R},(M,y,s)$ in the domain, \[G(cM,y,s)=cG(M,y,s).\]

\end{itemize}

Here (G1)--(G3) is the same as (F1)--(F3) with $x,t$ removed.
Notice that if $F,g$ satisfy (F1)--(F4), $G:=F^{x,t},h:=g(x,\cdot,t,\cdot)$ satisfy (G). %This pair of $(G,h)$ are going to be discussed in the study of the double boundary layer problems.

\smallskip

In view of the results in \cite{4} and Lemma \ref{lemmalpos} below, our method applies to the following operators
\[
F_\eps(D^2u,x,y,t,s)+G_\eps(Du,x,y,t,s)
\] 
with some suitable conditions, including $\eps^{2}F_\eps(\eps^{-2}M,x,y,t,s)\rightarrow F^{x,t}(M,y,s),\eps G_\eps(\eps^{-1}p,x,y,t,s)\rightarrow G^{x,t}(p,y,s)$
 locally uniformly as $\eps\rightarrow 0$, and
{$G_\eps(p,x,y,t,s)$ is uniformly Lipschitz continuous in $p$.}
If we assume the well-posedness and the interior homogenization of the general equations, then our homogenization results concerning oscillating boundary data in space-time domain hold as well. %problem and the interior hMany details and problems about the existence and uniqueness of solutions and homogenization in the interior region are involved if discussing these generations which may deviate our main purpose.

%\begin{flushleft}(O1)The lateral boundary $\partial_l\Omega$ is $\mathcal{C}^1$ in time direction and $C^2$ in space direction. \\(O2)(exterior cone condition)Let $\nu(x,t)$ be the unit interior normal at $(x,t)$ on the lateral boundary. There is a constant $\rho\in(0,1)$ that\[\cup_{0\leq s\leq \rho}B(x-s\nu(x,t),s\rho)\subset \Omega(t)^c.\](O3)(temporal variation)Define $f_d(t,x)=d\left(x,\Omega(t)\right)$. For all $t\in[0,T],x\in \mathbb{R}^d$, assume\[f_d(\cdot,x)\in\mathcal{H}^1\left([0,T],[0,\infty)\right).\](O4)We use $c(x,t)|_{\partial_l\Omega}$ to denote the speed of the boundary. We require $\frac{\partial c}{\partial t}$ exists in the weak sense and is uniformly bounded.\end{flushleft}

%It can be checked that the assumptions made in \cite{exist} about the domain are satisfied by condition (O).

\subsection{Previous results}
We use the notion of viscosity solutions throughout this paper which were originally introduced in \cite{15}. We also refer readers to \cite{vis, 9} concerning viscosity solutions to elliptic/parabolic equations with Dirichlet boundary conditions. %Then we introduce some backgrounds and previous results for future use. 

Consider the problem:
\begin{equation}\label{def vis}
\left\{
\begin{aligned}
  &  {\partial_t}{u}(x,t)-{F}(D^2 {u},x,t)=0 &\text{ in }&\quad\Omega_T,\\
&{u}(x,t)={g}(x,t) &\text{ on }&\quad\partial_p\Omega_T
\end{aligned}    \right.
\end{equation}
with $F,g$ satisfying (G1)(G2).

\begin{definition} (Viscosity solution)
\begin{enumerate}
\item[(i) ]
We say an upper semi-continuous function $u:\Omega_T\to \mathbb{R}$ is a subsolution to \eqref{def vis} if the following holds:
for any $(x,t)\in \partial_p\Omega_T$, $\limsup_{(y,s)\to (x,t)} u(y,s)\leq g(x,t)$;
for any smooth function $\phi$ such that $u-\phi$ has a local max at $(x_0,t_0)\in \Omega_T$, then ${\partial_t}{\phi}(x_0,t_0)-{F}(D^2 \phi,x_0,t_0)\leq 0$.

\item[(ii) ]
We say a lower semi-continuous function $u:\Omega_T\to \mathbb{R}$ is a subsolution to \eqref{def vis} if the following holds:
for any $(x,t)\in \partial_p\Omega_T$, $\liminf_{(y,s)\to (x,t)} u(y,s)\geq g(x,t)$;
for any smooth function $\phi$ such that $u-\phi$ has a local min at $(x_0,t_0)\in \Omega_T$, then ${\partial_t}{\phi}(x_0,t_0)-{F}(D^2 \phi,x_0,t_0)\geq 0$.

\item[(iii) ]
We say a continuous function $u:\Omega_T\to \mathbb{R}$ is a solution to \eqref{def vis} if it is both a subsolution and a supersolution.

\end{enumerate}
\end{definition}

It is known that viscosity solutions are stable under uniform convergence, see \cite{15,wf}.

\begin{lemma}
\label{lem stability}
Let $\{F_i,g_i\}$ be a sequence of operators and boundary data satisfying (G1)(G2). Suppose $(F_i,g_i)$ converge locally uniformly to $(F,g)$. Let $u_i$ be a sequence of bounded subsolutions (supersolutions) to \eqref{def vis} with $F,g$ replaced by $F_i,g_i$. Let $u^*$ be the upper half-relaxed limit of $u_i$, then $u^*$ is one subsolution to \eqref{def vis} (resp. the lower half-relaxed limit $u_*$ is one supersolution).
\end{lemma}

Comparison principle of fully nonlinear parabolic equations on a space-time domain with general Neumann type boundary condition is proved by Lundstr\"{o}m and \"{O}nskog \cite{exist} which generalizes the results on fixed domains by Dupuis and Ishii \cite{9}. It is by now classical that well-posedness for equation \eqref{eqns1} follows from Perron's method and comparison principle.

\subsubsection{Comparison, boundedness, and regularity results}

To state the comparison principle, we need to introduce the well-known \textit{Pucci's extremal operators}. The Pucci's operators with parameters $\Lambda>\lambda>0$ are defined as $\mathcal{P}^\pm(\lambda,\Lambda): \mathcal{M}\rightarrow \mathbb{R}$  such that
\[\mathcal{P}^+(\lambda,\Lambda)(M)=\Lambda \,{\text{Tr}}\, M_+-\lambda \,{\text{Tr}}\, M_-,\]
\[\mathcal{P}^-(\lambda,\Lambda)(M)=\lambda \,{\text{Tr}}\, M_+-\Lambda \,{\text{Tr}}\, M_-\]
where $M_\pm\geq 0$ are respectively the positive and negative part of $M$. We refer readers to \cite{elliptic} for more discussions. %The Pucci operator is one of the main tools in studying the fully nonlinear uniformly elliptic/parabolic equations.

Condition (G2) implies that if $u$ and $v$ satisfy, in the viscosity sense, 
\begin{align*}
    & {\partial_t}{u}(x,t)-{F}(D^2 {u},x,t)\leq 0 ,\\
    & {\partial_t}{v}(x,t)-{F}(D^2 {v},x,t)\geq 0 
\end{align*}
in $\Omega_T$, then $w:=u-v$ satisfies
\[
  %  \label{dif eqn}
    {\partial_t}{w}(x,t)-\mathcal{P}^+(D^2 {w})\leq 0
\]
in the viscosity sense.
We have the following comparison principle.

\begin{lemma}(Theorem 3.14 \cite{LW})
Let $u(x,t),v(x,t)$ be as the above. Then 
\[\sup_{\Omega_T}\,(u-v)_+=\sup_{\partial_p\Omega_T}(u-v)_+.\]
\end{lemma}

Due to condition (G2), 
\[\mathcal{P}^-(M)-|F(0,x,t)|\leq F(M,x,t)\leq \mathcal{P}^+(M)+|F(0,x,t)|.\] 
If for some bounded function $h(x,t)$, a continuous function $u$ satisfies
\begin{equation*}
\frac{\partial}{\partial t}u-\mathcal{P}^-(D^2 u)+C_1|Du|-h\geq 0,
\end{equation*}
\begin{equation*}
\frac{\partial}{\partial t}v-\mathcal{P}^+(D^2 v)-C_1|Dv|-h\leq 0
\end{equation*} 
in the viscosity sense in some domain, then we say $u\in S(\lambda,\Lambda,h)$. 

We have the following result showing the boundedness of solutions.

\begin{proposition}(Corollary 3.20 \cite{LW})
{Let $u\in S(\lambda,\Lambda,h)$ in $\Omega_T$. Then there exits $C$ such that}
\[\|u\|_{L^\infty(\Omega_T)}\leq \|u\|_{L^\infty(\partial_p\Omega_T)}+C\|h\|_{L^{d+1}(\Omega_T)}.\]
In particular, solutions to \eqref{def vis} are bounded.
\end{proposition}

\smallskip

Regularities of viscosity solutions to fully nonlinear parabolic equations are studied by Wang \cite{LW, LW2} and Imbert et al. \cite{PARA}. The following two theorems provide the interior regularity and the boundary regularity of solutions respectively. 
%For simplicity we write $Osc(r)(u)=Osc_{S_r}(u)   $ where $S_r$ is a parabolic cylinder.

\begin{theorem}\label{LW0}
(Theorem 4.19 \cite{LW}) Let $u\in S(\lambda,\Lambda,h)$ in a parabolic cylinder $S_1$. There exist $\beta\in(0,1),C\geq 1$ such that for all $r<1$,
\[\osc_{(S_r)} (u)\leq C(r^\beta {\osc_{(S_1)}(u)}+r\|h\|_{d+1,S_1})\]
and
\[\|u\|_{C^\beta(S_\frac{1}{2})}\leq C(\|u\|_{\infty,S_1}+\|h\|_{d+1,S_1}).\]

\end{theorem}

\begin{corollary}\label{cor LW}
Let $R\geq N\geq 1$. Suppose a bounded function $u$ solves \eqref{def vis} in $S_R$ and $F,g$ satisfy (G). Then
\[\osc_{(S_N)}(u)\leq C\left({N}/{R}\right)^\beta.\]
\end{corollary}
\begin{proof}
Let $v(x,t)=u(Rx,R^2t)$ which then solves $ {\partial_t}{v}(x,t)-{F}(D^2 {v},\frac{x}{R},\frac{t}{R^2})=0$ in $S_1$. Here we used the homogeneity condition (G4). Also by (G4), $v\in S(\lambda,\Lambda,0)$ and thus Theorem \ref{LW0} concludes the proof.
\end{proof}

\begin{theorem}\label{cont}
(Theorem 2.5 \cite{LW2}) Let $u\in S(\lambda,\Lambda,h)$. If $u|_{\partial_p\Omega}$ is H\"{o}lder continuous at $(x_0,t_0)\in \partial_p\Omega$ and $\partial_p\Omega$ is Lipschitz at $(x_0,t_0)$, then $u$ is $C^\beta$ at $(x_0,t_0)$ for some $\beta$.
\end{theorem}

Though \cite{LW2} only considered stationary domain, the regularity result for space-time domains with $\mathcal{C}^1$ boundary can be proved in the same way.

\subsubsection{Interior homogenization}\label{223}

Interior periodic homogenization result for fully nonlinear parabolic equations is proved Marchi \cite{inthomo}. From the results (see Propositions 2.1, 3.2 \cite{inthomo}), we have the existence of the homogenized operator.%, the definition of which can be seen in the following theorem. 
\begin{theorem}\label{homo}
Suppose $F$ satisfies conditions (F1)--(F3). Then there exists a unique operator $\bar{F}:\mathcal{M}\times \mathbb{R}^{d+1}\to\mathbb{R}$ satisfying (F1)--(F2) such that the following happens. Let $\{u^\eps\}$ be a sequence of uniformly bounded subsolutions to
\begin{equation}
    \label{main eqn}
    \partial_t u^\eps(x,t)-F(D^2u^\eps,x,\frac{x}{\eps},t,\frac{t}{\eps^2})\leq 0
\end{equation} 
in a bounded open set $O\subset\mathbb{R}^{d+1}$. Then $u^*:=\limsup^* u^\eps$ satisfies
\begin{equation}
    \label{homogenized eqn}
   {\partial_t}u^*(x,t)-\bar{F}(D^2 u^*,x,t)\leq 0    \text{ in } O.
\end{equation}
Similarly if $u^\eps$ are supersolutions to \eqref{main eqn}, then $u_*:=\liminf_* u^\eps$ is a supersolution to \eqref{homogenized eqn}.
\end{theorem}

\subsection{Localization lemmas}

For any unit vector $\nu$, let \[P_{\nu}:=\{ (x,t)\in\mathbb{R}^{d+1}\, |\text{ }x\cdot \nu\geq 0 \}.\] 
Denote
\begin{equation}
    \label{notation xnu}
x_\nu:=(x\cdot \nu), \quad  x':=x-x_\nu\, \nu,
\end{equation} 
and
\[Q_L:=\{(x,t)\in\mathbb{R}^{d+1}\, |\text{ } 0< x_\nu< L, |t|<  L^2,  |x'|< L\}.\]

Since in later sections the parabolic operator with a first order term in $\mathcal{P}_\nu$ is used, we show the following localization lemma for Pucci's operator with a drift. 

\begin{lemma}\label{lemmalpos}
{Suppose that for some $\eps\geq 0$ and $C_0>c_0>0$, $w(x,t)$ satisfies}
\begin{align}&
\frac{\partial}{\partial t}w-\mathcal{P}^+(D^2 w)-\eps |Dw|\leq 0 & \text{ in  }& \quad  Q_{ L} ,\label{para drift}\\
&w(x,t)\leq c_0 &  \text{ on }& \quad \overline{Q}_{ L}\cap\partial P_{\nu} ,\nonumber\\
&w(x,t)\leq C_0 & \text{ on }& \quad \overline{Q}_{ L}.\nonumber
\end{align}
Then there exist constants $c=\frac{\lambda(d-1)}{4d}$ and $C>0$ depending only on $\lambda,\Lambda,d,C_0$ such that for any $R<L$, if $L\eps\leq c$ we have
\[w(x,t)\leq  c_0+C{R}/L\quad\text{ for } (x,t)\in Q_{ R}.\]

\end{lemma}

\begin{proof}
Let $C_1=C_0\max\{1,\frac{\lambda}{2\Lambda(d-1)},\frac{2}{\Lambda(d-1)}\}$. We construct the following barrier,
\[\phi(x,t)= \frac{C_1}{L^{2}}(|x'|^2-\frac{2\Lambda}{\lambda}(d-1)x_\nu^2)+\frac{4C_1\Lambda}{L\lambda}(d-1)x_\nu+c_0+\frac{C_0}{L^4}t^2.\]
It follows from direct computations that in $Q_L$
\begin{align*}
    \frac{\partial}{\partial t}\phi-\mathcal{P}^+(D^2 \phi)&\geq \frac{2C_0}{L^4}t-\frac{\Lambda C_1}{L^2}\,{\text{Tr}} \,(D^2|x'|^2)+\frac{C_1\lambda}{L^2} \,{\text{Tr}}\, (D^2 \frac{2\Lambda}{\lambda}(d-1)x_\nu^2)\\ &\geq-{2C_0}L^{-2}+2C_1\Lambda(d-1)L^{{-2}},
\end{align*}
where we used the definition of the Pucci's operator and $(x,t)\in Q_L$.
Notice
\begin{align*}
    |D\phi|&=2C_1L^{-2}\left(|x'|^2+\left(\frac{2\Lambda}{\lambda}(d-1)\right)^2(x_\nu-L)^2\right)^{\frac{1}{2}}\\
    &\leq 2C_1L^{-1}\left(1+\left(\frac{2\Lambda}{\lambda}(d-1)\right)^2\right)^{\frac{1}{2}}\quad \text{( since $|x'|\leq L$, $|x_\nu|\leq L$ )}\\
    &\leq 4C_1dL^{-1}{\lambda}^{-1}{\Lambda}.
\end{align*}
Therefore inside $ Q_L$, we have
\begin{align*}
    &\quad \frac{\partial}{\partial t}\phi-\mathcal{P}^+(D^2 \phi)-\eps | D\phi|\\
%    & \geq -2{C_0}L^{-2}+2C_1(d-1)\Lambda L^{-2}-\eps c_1+\eps^3 c_2t||D\phi|\\
    & \geq -2{C_0}L^{-2}+2C_1(d-1)\Lambda L^{-2}-4\eps C_1L^{-1}d\,\frac{\Lambda}{\lambda}\\
    &\geq C_1\Lambda L^{-2}\left((d-1)-4\eps Ld\lambda^{-1}\right) \quad \text{ (since $C_1\Lambda(d-1)\geq 2C_0$) }.
\end{align*}
If $L\eps\leq  \frac{\lambda(d-1)}{4d}$, the above is non-negative and so $\phi$ is a supersolution.

Also it is not hard to verify that $\phi\geq w$ on the boundary:
when $|x'|=L$ or $x_\nu=L$, we have
\[\phi(x,t)\geq C_1 \min\{1,\frac{2\Lambda}{\lambda}(d-1)\}\geq C_0\geq w(x,t);\] when $t=\pm L^2$, 
\[\phi(x,\pm L^2)\geq C_0\geq w(x,\pm L^2);\] finally when $x_\nu=0$, 
$\phi\geq c_0\geq w$ by the assumption.
From the definition of viscous solution, we know $w(x,t)\leq \phi(x,t)$ in $ Q_L$. In particular for $(x,t)\in Q_R$, we derive
\begin{align*}
    w(x,t)&\leq \phi(x,t)\\
    &\leq \frac{C_1}{L^2}|x'|^2+C_1(d-1)\frac{2\Lambda}{\lambda}\frac{x_\nu}{L}(2-\frac{x_\nu}{L})+c_0\\
   & \leq c_0+C{R} /L
\end{align*}
where $C$ is a constant only depending on $\lambda,\Lambda,d,C_0$.
\end{proof}

%\begin{lemma}\label{lemmalpos}
%Let $C_0>c_0>0$ and $L>R>0$. Suppose $w$ satisfies (in the viscosity sense) the following equation 
%\begin{equation*}
%\left\{\begin{aligned}&
%\frac{\partial}{\partial t}w- \mathcal{P}^+(D^2w)\leq 0 & \quad \text{ in  }&\quad  Q_L ,\\
%&w(x,t)\leq c_0 & \quad \text{ on }&\quad  \partial P_{\nu}\cap \overline{ Q}_L,\\
%&w(x,t)\leq C_0 & \quad \text{ on }&\quad  \overline{ Q}_L.
%\end{aligned}\right.\end{equation*}Then there exists a constant $C(\lambda,\Lambda,d,C_0)$ such that\[w(x,t)\leq c_0+C\frac{R}{L} \quad \text{     in } Q_R.\]\end{lemma}

\begin{corollary}\label{cols}
(Comparison Principle for bounded solutions in half-space) Suppose $F:\mathcal{M}\times\mathbb{R}^{d+1}\rightarrow \mathbb{R}$ is uniformly elliptic and there are two bounded solutions $v_1,v_2$ of
\[\frac{\partial}{\partial t}v(x,t)- F(D^2v,x,t)= 0\quad \text{ in }\mathcal{P}_\nu.\]
Then
\[\sup_{(x,t)\in P_{\nu}}\left(v_1(x,t)-v_2(x,t)\right)_+\leq \sup_{(x,t)\in \partial{P_{\nu}}}\left(v_1(x,t)-v_2(x,t)\right)_+.\]
\end{corollary}
\begin{proof}

Suppose, to the contrary, there exists $(x_0,t_0)\in \mathcal{P}_\nu$ such that \[v_1(x_0,t_0)-v_2(x_0,t_0)> \sup_{(x,t)\in \partial{P_{\nu}}}\left(v_1(x,t)-v_2(x,t)\right)_+=:c_0.\] 
Write $w=v_1-v_2$ and assume $w\leq C_0$. 
Since $F$ is elliptic, $w$ satisfies $\frac{\partial}{\partial t}w- \mathcal{P}^+(D^2w)\leq 0.$
Take $R$ to be large enough such that $(x_0,t_0)\in Q_R$. By Lemma \ref{lemmalpos}, for any $L>R$ we have
$w(x,t)\leq c_0+C{R}/L$ in $Q_R$. If taking $L\to \infty$, we get $(v_1-v_2)(x_0,t_0)\leq \sup_{Q_R}w\leq c_0$ which contradicts with the assumption.
\end{proof}

\begin{remark}
Let us point out that for the equation \eqref{para drift} in the domain $\mathcal{P}_\nu$, we may not have uniqueness among bounded continuous functions when $\eps>0$. Indeed, consider
\begin{equation*}
\left\{\begin{aligned}&
w_t- w_{xx}-\eps w_x=0 &\text{ for }x>0,t\in\mathbb{R},\\
&    w(0,t)=1. &
\end{aligned}
\right.
\end{equation*}
Then $w=1,w=e^{-{\eps}x}$ are obviously two solutions.
Though uniqueness fails, Lemma \ref{lemmalpos} implies that if the first order term is small enough ($L\eps<c$), we can still bound the difference of two solutions in the interior by the difference of their boundary values plus a small error.
\end{remark}
%Now we give the following lemma which will be used 

The claim of Lemma \ref{lemmalpos} is still valid when the boundary of the domain is (only) varying in time. The following corollary is stated in a version which can be directly applied in Section 5.

\begin{corollary}\label{cmp coro}
Let $f:\mathbb{R}\to\mathbb{R}$ be a $\mathcal{C}^1$ function with $f(0)=0$, and let $\nu\in\mathbb{S}^{d-1}$, $\eps\in (0,1)$.
Using the notation \eqref{notation xnu}, denote
\[
D(f,L):=\{(x,t)\,|\,\sqrt{\eps}f(t)<x_\nu<L+\sqrt{\eps}f(t),\,|x'|< L,\,|t|\leq L^2\}
\] 
and 
\[
\partial_l D(f,L):=\{(x,t)\,|\,x_\nu= \sqrt{\eps}f(t),\,|x'|< L,\,|t|< L^2\}.
\] 
Assume that the operator $F_\eps$ satisfies (G1)--(G3) and $|F_\eps(0,\cdot,\cdot)|_\infty\leq C_1$ for some $C_1>0$. Suppose that $w$ is a bounded solution to
\[\partial_t w-\eps F_\eps\left(\eps^{-1} D^2w,\frac{x}{\sqrt{\eps}},\frac{t}{\eps}\right)-\sqrt{\eps}\,\xi\cdot Dw= 0\]
in $D$ for some bounded vector field $\xi=\xi(x,t)$, and for any $\delta\in (0,1)$, $|w|\leq \delta $ on $\partial_l D(f,L)$. Then there exists $c=c(C_1)>0$ and $R=R(\delta)<L$ such that if \[L\sqrt{\eps}(|\xi|_\infty+|f'|_\infty)<c\quad\text{ and }\quad  L^2\eps\leq c\delta,\] 
we have
\[|w|\leq 3\delta \quad \text{ in } D(f,R).\]
\end{corollary}

\begin{proof}
Let
$\tilde{w}:=w(x+\sqrt{\eps}f(t)\nu,t)$ and then $\tilde{w}$ satisfies
\[\partial_t \tilde{w}-\eps F_\eps\left(\eps^{-1}D^2\tilde{w},\frac{x+\sqrt{\eps}f(t)\nu}{\sqrt{\eps}},\frac{t}{\eps}\right)-\sqrt{\eps}\,\xi\cdot D\tilde{w}-\sqrt{\eps}f'(t)\nu\cdot D\tilde{w}= 0.\]
Let us denote the above operator $\eps F_\eps(\eps^{-1}M,...)$ as $F_{\eps,f}(M,x,t)$.
In view of (G2) and boundedness of $F_{\eps,f}(0,x,t)$, we have
\[
\mathcal{P}^-(M)-C_1\eps\leq F_{\eps,f}( M,\cdot,\cdot)\leq   \mathcal{P}^+(M)+C_1\eps,
\]
which implies that that in $Q_L$,
\[\partial_t \tilde{w}-\mathcal{P}^+(D^2 \tilde{w})-c_1\sqrt{\eps}| D\tilde{w}|-C_1\eps\leq 0,\]
\[\partial_t \tilde{w}-\mathcal{P}^-(D^2 \tilde{w})-c_1 \sqrt{\eps} |D\tilde{w}|+C_1\eps\geq 0\]
where $c_1:=|\xi|_\infty+|f'|_\infty$.
To get rid of the zero order term, let $w_\pm:=\tilde{w}\pm C_1\eps t$. Applying Lemma \ref{lemmalpos} to both $w_+$ and $-w_-$ yields
\[w_\pm\leq \delta+C_1\eps L^2+C{R}/L\]
in $Q_R$ if $L\sqrt{\eps} c_1$ is small enough. We obtain
\[|w|\leq \delta+C_1\eps (L^2+R^2)+C{R}/L\]
in $D(f,R)$ and the conclusion follows if taking $R=\frac{L\delta}{C}$ and using the assumption that $L^2\eps$ is small.
\end{proof}

\section{The Boundary Layer Problems}\label{lateralb}

In this section, for each lateral boundary point we will analyze the cell problems and study their boundary layer limits. In the end, we will identify the lateral homogenized boundary data to be the boundary layer limit.

\subsection{Irrational normal case}\label{3.1}

Before studying the cell problems, we recall Lemma 2.7 in \cite{kim} which shows that any points on half-planes with irrational normal $\nu$ can be approximated in some sense by a point in $\mathbb{Z}^d$.

\begin{lemma}\label{lemm discpcy}
Let $\nu\in \mathbb{S}^{d-1}$ be irrational. There exists a function $\omega_\nu:\mathbb{Z}^+\to\mathbb{R}^+$ satisfying $\omega_\nu(N)\to 0$ as $N\to \infty$ such that the following holds. For any $x_0\in\mathbb{R}^d$ and then any {$x\in \{y\,|\,(y-x_0)\cdot\nu=0\}$} and any $N>1$, there exists $m\in\mathbb{Z}^d$ such that
\[|x-m|\leq N, \quad dist(m,\partial P_\nu)\leq \omega_\nu(N).\]
\end{lemma}

The function $\omega_\nu$ is given through a discrepancy function, see \cite{kuipers1975uniform,wf}.
 
%\begin{definition}For any $x\in [0,1]$, let $x_j=jx-\lfloor{jx}\rfloor$ where $\lfloor{x}\rfloor$ denotes the largest integer less than or equal to $x$. Define \end{definition}

\smallskip

Now fix $(x_0,t_0)\in \partial_l\Omega$ with $t_0>0$ and write $\nu$ as the interior spatial normal at $(x_0,t_0)$. 
In the case when $\nu$ is irrational, consider
\[
v^\eps(x,t):=u^\eps(x_0+\eps x,t_0+\eps^2 t).
\]
Then the corresponding domain $\Omega^\eps=\bigcup_{t\in\bbR}(\eps^{-1}(\Omega(\eps^2 t+t_0)-x_0)\times\{t\})$ converges to $P_\nu$ locally uniformly in Hausdorff distance. 
As explained in the introduction, by passing to a subsequence of $\eps\to 0$, we assume $\eps^{-1}x_0$, $\eps^{-2}t_0 $ converge to $z,\tau$ in $\mathbb{R}^d/\mathbb{Z}^d,\mathbb{R}/\mathbb{Z}$ respectively. Fix one such $z,\tau$ for a moment. Formally we derive
\begin{equation*}
\text{({I})}\quad \left\{
\begin{aligned}&
\frac{\partial}{\partial t}v^{z,\tau}(x,t)- F^0(D^2v^{z,\tau},x+z,t+\tau)=0 & \quad \text{ in }&\quad   P_{\nu},\\
&v^{z,\tau}(x,t)=g^0(x+z,t+\tau)& \quad \text{ on }&\quad   \partial P_{\nu}
\end{aligned}
\right.
\end{equation*}
where $g^0(y,s):= g(x_0,y,t_0,s)$, $F^0(M,y,s):=F^{x_0,t_0}(M,y,s)$.
We call this the cell problem of irrational normal case.
Since $u^\eps$ is uniformly bounded, we solve for bounded solutions of (I). Lemma \ref{lemmalpos} and Perron's method gives the well-posedness of the problem.

In the following proposition, we show the existence of the boundary layer limit of (I).

%We claim that the above limit exists and is independent of $x,t,z,\tau$. The main step is: shift the solutions in both space and time direction and then apply comparison principle. The proposition is analogous to Lemma 3.1 and Lemma 5.2 in \cite{wf}.
\begin{lemma}\label{irgbar}
Suppose $\nu$ is irrational, and $v^{z,\tau}(x,t)$ is the unique solution to equation (I). {Then for any $(x,t)\in P_\nu$,}
\[
\lim_{R\rightarrow \infty} v^{z,\tau}(x+R\nu,t)
\]
exists and the limit is independent of $x,t,z,\tau$ which is denoted as $\varphi(F^0,g^0,\nu)$. Moreover, for any $N\geq 1$, we have
\begin{equation}
    \label{rate cvg irr}\left|v^{z,\tau}(x+R\nu,t)-\varphi(F^0,g^0,\nu)\right|\leq C\left(({N}/{R})^\beta+w_{\nu}(N)^\beta\right)
\end{equation}
for some constant $C=C(\nu,g^0,F^0)$.
Here $\beta$ is from Theorem \ref{LW0} and $\omega_\nu(\cdot)$ is given in Lemma \ref{lemm discpcy}.
\end{lemma}

\begin{proof}
The lemma is analogous to Lemma 3.1 and Lemma 5.2 in \cite{wf}. First we show that the limit exists and is independent of $x,t$. By Lemma \ref{lemm discpcy}, for any $(x,t)\in \partial P_\nu$ and $ N\geq 1$, there exists $m_1\in \mathbb{Z}^d$ such that
\[|x-m_1|\leq N \text{ and } dist(m_1,\partial P_\nu)\leq \omega_\nu(N).\]
Also we can select $m_2\in \mathbb{Z}$ be such that $|t-m_2|\leq 1$.
First let us assume $h:=m_1\cdot\nu=dist(m_1,\partial P_\nu)>0$. Then consider
\[\tilde{v}^{z,\tau}(x,t)=v^{z,\tau}(x-m_1,t-m_2),\]
which satisfies

\begin{equation*}
\quad \left\{
\begin{aligned}
&\frac{\partial}{\partial t}\tilde{v}^{z,\tau}(x,t)- F^0(D^2\tilde{v}^{z,\tau},x+z,t+\tau)=0 & \quad \text{ for }& (x,t)\in P_{\nu}+h\nu,\\
&\tilde{v}^{z,\tau}=g^0(x+z,t+\tau)& \quad \text{ for }& (x,t)\in \partial P_{\nu}+h\nu.
\end{aligned}
\right.
\end{equation*}
By H\"{o}lder estimate and continuity of the boundary data, 
$|\tilde{v}^{z,\tau}-v^{z,\tau}|\leq C h^\beta$ on $\partial P_{\nu}+h\nu$. Since $v^{z,\tau}$ and $\tilde{v}^{z,\tau}$ satisfy the same equation, we compare the two in $P_\nu+h\nu$ to get 
\[|\tilde{v}^{z,\tau}(R\nu,0)-v^{z,\tau}(R\nu,0)|\leq C h^\beta\leq C\omega(N)^\beta.\]
By Corollary \ref{cor LW}, 
\begin{align*}
    |v^{z,\tau}(x+R\nu,t)-v^{z,\tau}(R\nu,0)|&\leq |v^{z,\tau}(x+R\nu,t)-v^{z,\tau}(m_1+R\nu,m_2)|+|\tilde{v}^{z,\tau}(R\nu,0)-v^{z,\tau}(R\nu,0)|\\
&\leq \osc_{(\{(x',t')\,|\,|(x',t')-(x+R\nu,t)|\leq CN\})}\left(v^{z,\tau}\right)+C\omega_\nu(N)^\beta&\\
&\leq C\left(({N}/{R})^\beta+\omega_\nu(N)^\beta\right).
\end{align*}
By maximum principle, 
\begin{equation}
    \label{osc vztau}
    \osc_{( P_\nu)}v^{z,\tau}(\cdot+R\nu)\leq \osc_{(\partial P_\nu)}v^{z,\tau}(\cdot+R\nu)\leq C\left(({N}/{R})^\beta+\omega_\nu(N)^\beta\right).
\end{equation}
Thus $\lim_{R\rightarrow \infty} v^{z,\tau}(x+R\nu,t)$ exists and is independent of the choice of $x,t$. If $h=-dist(m_1,\partial P_\nu)<0$, we only need to compare $\tilde{v}^{z,\tau}$ and $v^{z,\tau}$ in $P_\nu$ instead of $P_\nu+h\nu$ to complete the proof.

\smallskip

Next we show that the limit is also independent of $z,\tau$. Suppose $v^{0}$ is the solution to equation (I) with $\{z,\tau\}$ replaced by $\{0,0\}$. Again let $m_1\in\mathbb{Z}^d$ be such that 
\[|z+m_1|\leq N, dist(m_1+z,\partial P_\nu)\leq \omega_\nu(N).\]
Let $h:=(m_1+z)\cdot \nu$ and without loss of generality we assume $h\geq 0$.
Then $\tilde{v}(x,t):=v^{z,\tau}(x-z-m_1,t-\tau)$ satisfies
\begin{equation*}
\quad \left\{
\begin{aligned}
&\frac{\partial}{\partial t}\tilde{v}(x,t)- F^0(D^2\tilde{v},x,t)=0 & \quad \text{ for }& (x,t)\in P_{\nu}+h\nu,\\
&\tilde{v}=g^0(x,t)& \quad \text{ for }& (x,t)\in \partial P_{\nu}+h\nu.
\end{aligned}
\right.
\end{equation*}
Similarly as above, we compare $\tilde{v}$ with $v^{0}$ to find that
\begin{equation}
    \label{osc ztau}|v^{z,\tau}(R\nu,0)-v^{0}(R\nu,0)|\leq C\left(({N}/{R})^\beta+\omega_\nu(N)^\beta\right)
\end{equation}
which shows that the limit is indeed independent of $z,\tau$. Finally \eqref{osc vztau} and \eqref{osc ztau} give the rate of convergence \eqref{rate cvg irr} and the constant $C$ only depends on $\nu,g^0,F^0$.
\end{proof}
\begin{remark}

In general, in the lemma we cannot remove the condition that $\nu$ is irrational.
But if $\{x\,|\,(x-x_0)\cdot \nu=0\}$ passes through the original point, it is not hard to check that different $z$'s just cause shifts along hyperplane $\partial P_\nu$ and the limit ${\lim}_{R\to\infty} v^{z,\tau}(x+R\nu,t)$ is then again independent of $z,\tau$ (see problem 1.1 in \cite{4}). 
\end{remark}

\subsection{Rational normal case}\label{3.2}

Fix $(x_0,t_0)\in\Gamma_2$ with rational normal $\nu$. Denote $c(x,t)$ as the interior normal speed of $(x,t)\in \partial_l\Omega$. Due to the local spacial flatness assumption, $c=c(t)$ in $ S_r(x_0,t_0)\cap \partial_l\Omega=S_r(x_0,t_0)\cap \Gamma_2$ for some $r>0$ and it is nonzero.
As described in the introduction, we set
\[v^\eps(x,t):=u^\eps(x_0+x_\eps+c_0 t_\eps\nu,t_0+ t_\eps)\]
where $x_\eps=\eps^\frac{1}{2}y+\eps x$, $t_\eps=\eps s+\eps^2 t$ and $c_0=c(t_0)$. 
The equation becomes
\[
\partial_tv^\eps(x,t)-\eps c_0\nu\cdot Dv^\eps-\eps^2 F(\eps^{-2}D_x^2v^\eps,x_0+x_{\eps},\eps^{-1}(x_0+x_\eps+c_0 t_\eps\nu),t_0+t_\eps,\eps^{-2}(t_0+t_\eps)) =0  \]
By passing to a subsequence of $\eps\rightarrow 0$, we assume $(\eps^{-1}x_0+\eps^{-\frac{1}{2}}y)$, $(\eps^{-2}t_0+{\eps}^{-1}s)$ converges to $z,\tau$ in $\mathbb{R}^d/\mathbb{Z}^d,\mathbb{R}/\mathbb{Z}$ respectively. Then formally, we derive our first boundary layer problem:
\begin{equation*}
\text{(R1)}\quad \quad\left\{
\begin{aligned}&
\frac{\partial}{\partial t}v^{z,\tau}_{s}(x,t)- F^0(D^2v^{z,\tau}_{s},x+z+{c_0 s \nu},t+\tau)=0 &  \text{ in }& \quad  P_\nu,\\
&v^{z,\tau}_{s}(x)=g^0(x+z+ c_0 s\nu,t+\tau ) &  \text{ on }&\quad   \partial P_\nu.
\end{aligned}
\right.
\end{equation*}
Here $g^0,F^0$ are the same as before (as in (I)). By Perron's method and comparison principle, there exists a unique bounded solution $v^{z,\tau}_{s}$. For simiplicity of notation, we may write $c=c_0$ later. We prove the following lemma.

\begin{lemma}\label{bd lyer lmt R1}
Let $v^{z,\tau}_{s}$ be the unique bounded solution of (R1) for some $z,\tau,s$. Then for any $(x,t)$, ${\lim}_{R\rightarrow \infty}v^{z,\tau}_{s}(x+R\nu,t)$ exists and it only depends on $z\cdot\nu+cs,\nu,F^0,g^0$.
Let us write the limit as $\chi^{z}_{\nu}(cs)=\chi^{z}_{\nu,s}(F^0,g^0,\nu,z\cdot\nu+cs)$. Then for some constant $N$ only depending on $\nu$ and for all $x,t\in P_\nu$, we have
\begin{equation}
    \label{r1 rt est}
    |v_s^{z,\tau}(x+R\nu,t)-\chi^{z}_{\nu}(cs)|\leq C\left({N}/{R}\right)^\beta.
\end{equation}
Moreover, if $\nu=\frac{\hat{\nu}}{|\hat{\nu}| }$ with irreducible $\hat{\nu}\in \mathbb{Z}^d$, then $\chi^{z}_{\nu}(c\,\cdot)$ is a periodic function with periodicity $\frac{1}{|c\,\hat{\nu}|}$.
\end{lemma}
\begin{proof}
Fix any $(x,t)\in P_\nu$. Since $\nu$ is rational, we can take $m_1\in \mathbb{Z}^d,m_2\in \mathbb{Z}$ such that 
\[|x-m_1|\leq N(\nu),\quad m_1\cdot \nu=0,\quad |t-m_2|\leq 1\]
where $N(\nu)$ is a constant only depending on $\nu$.
Since
\[
\tilde{v}^{z,\tau}_{s}(x,t)={v}^{z,\tau}_{s}(x-m_1,t-m_2)
\] 
is also a solution to (R1) with the same boundary data, by uniqueness $\tilde{v}^{z,\tau}_{s}:={v}^{z,\tau}_{s}$. By Corollary \ref{cor LW}
\begin{equation}\label{r1 uniformins}
|{v}^{z,\tau}_{s}(x+R\nu,t)-{v}^{z,\tau}_{s}(R\nu,0)|=|{v}^{z,\tau}_{s}(x-m_1+R\nu,t-m_2)-{v}^{z,\tau}_{s}(R\nu,0)|\leq C({N}/{R})^\beta.
\end{equation}
Then similarly as done in Lemma \ref{irgbar}, we conclude with the help of the comparison principle that the boundary layer limit
\[\chi^{z,\tau}_{\nu,s}:={\lim}_{R\rightarrow \infty}v^{z,\tau}_{s}(x+R\nu,t)\]
exists and it is independent of $x,t$.

Note that ${v}^{z+cs\nu,\tau}_{0}={v}^{z,\tau}_{s}$. Thus we have $\chi^{z+cs\nu,\tau}_{\nu,0}=\chi^{z,\tau}_{\nu,s}$.
Now write $z=z'+z_\nu $ with $z_\nu=(z\cdot\nu)\nu$. We are going to show the independence of $\chi^{z,\tau}_{\nu,s}$ on $z',\tau$. Set 
\[
\tilde{\tilde{v}}^{z,\tau}_{s}(x,t):={v}^{z,\tau}_{s}(x-z',t-\tau).
\] 
which is then the unique solution to equation (R1) in $P_\nu$ with $z,\tau$ replaced by $z_\nu,0$ respectively. Since the boundary layer limit is independent of $x,t$, we have
\[\chi^{z_\nu,0}_{\nu,s}={\lim}_{R\rightarrow \infty}v^{z,\tau}_{s}(x-z'+R\nu,t-\tau)={\lim}_{R\rightarrow \infty}v^{z,\tau}_{s}(x+R\nu,t)=\chi^{z,\tau}_{\nu,s}.\]
So $\chi^{z,\tau}_{\nu,s}$ only depends $z\cdot\nu+cs,\nu$, and we might denote $\chi^{z,\tau}_{\nu,s}$ as $\chi^{z}_{\nu}(cs)$. 
The rate of convergence estimate \eqref{r1 rt est} follows from  \eqref{r1 uniformins}.

Furthermore, by the geometry there is $i\in \mathbb{Z}^d$ such that
$y:=i-\frac{\nu}{|\hat{\nu}|}\in \partial P_\nu$. Thus
\[cs\nu-y=c\left(s+(c|\hat{\nu}|)^{-1}\right)\nu\quad\text{ mod }\mathbb{Z}^d.\]
By periodicity and $y\in \partial P_\nu$, we have $v_s^{z,\tau}(x-y,t)=v_{s+(c|\hat{\nu}|)^{-1}}^{z,\tau}(x,t)$. Therefore the boundary layer limit satisfies $\chi^{z}_{\nu}(cs)=\chi^{z}_{\nu}(cs+(|\hat{\nu}|)^{-1})$.

\end{proof}

Then we look at a larger scale by considering
\[w^{\eps}(y,s): = u^\eps(x_0+\eps^\frac{1}{2}y+c\eps  s\nu,t_0+\eps  s)\]
which solves
\begin{equation}
    \label{derive R2}
    \begin{aligned}
    \partial_sw^\eps&-(\eps^\frac{1}{2} c)\nu\cdot D_yw^\eps-\\
    &\eps F(\eps^{-1}D_y^2w^\eps,x_0+\eps^\frac{1}{2}y+c\eps s\nu,\eps^{-1}(x_0+\eps^\frac{1}{2}y)+cs\nu),t_0+\eps s,\eps^{-2}(t_0+\eps s)) =0
    \end{aligned}
\end{equation}
As described in the introduction, we will use $\chi^{z}_{\nu}(c\,\cdot)$ as the boundary data for the second boundary layer problem. 
Suppose $\eps^{-\frac{1}{2}}x_0\rightarrow z_0,\eps^{-1}t_0\to\tau_0$ along a subsequence of $\eps\to 0$. This and \eqref{derive R2} suggest us to consider the following operator
\begin{align*}
     \partial_s w^\eps-\lim_{\eps\to 0} F^0(D_y^2 w^\eps,\eps^{-\frac{1}{2}}y+z_0+ cs \nu,\eps^{-1}s+\tau_0)=0
\end{align*}
which is itself a homogenization problem.

By Theorem \ref{homo}, there exists a unique homogenized operator, written as $\bar{F}^{0}(M,z_0+cs\nu,\tau_0)$, associated to $ F^0(M ,\eps^{-\frac{1}{2}}y+z_0+ cs \nu,\eps^{-1}s+\tau_0)$. We need the following lemma.

\begin{lemma}\label{homobarF}
For each $M$, the homogenized operator $\bar{F}^{0}(M,z+ct \nu,\tau)$ is independent of $z+ct\nu,\tau$. 
\end{lemma}
\begin{proof}
Let $\bar{F}^1(M,z+ct\nu,\tau)$ and $\bar{F}^0(M)$ be the homogenized operators for $F^0(M,\eps^{-1}x+z+ct\nu,\eps^{-2}t+\tau)$ and $F^0(M,\eps^{-1}x,\eps^{-2}t)$ respectively in the sense of Theorem \ref{homo}. We only need to show $\bar{F}^1=\bar{F}^0$. 

By Theorem 2.1 \cite{inthomo}, Marchi showed that there exists one unque real number $\bar{F}^1$ such that the following problem has a solution:
\begin{equation*}
\left\{
\begin{aligned}&
\frac{\partial}{\partial \sigma }\rho(\xi,\sigma)- F^0(D_\xi^2\rho+M,\xi+z+ct\nu,\sigma+\tau)=\bar{F}^1 \quad  \text{for } (\xi,\sigma)\in \mathbb{R}^{d+1},\\
&\rho=\rho(\xi,\sigma)\quad  \text{ is periodic in } \xi,\sigma
\end{aligned}
\right.
\end{equation*}
and the $\bar{F}^1$ is defined to be $\bar{F}^1(M,z+ct\nu,\tau)$. We notice that if $\rho(\xi,\sigma)$ is a solution to the above problem, then $\rho(\xi-z-ct\nu,\sigma-\tau)$ is a solution to
\begin{equation*}
\left\{
\begin{aligned}
&\frac{\partial}{\partial \sigma }\rho(\xi,\sigma)- F^0(D^2\rho+M,\xi,\sigma)=\bar{F}^1 \quad  \text{for } (\xi,\sigma)\in \mathbb{R}^{d+1},\\
&\rho=\rho(\xi,\sigma)\quad  \text{ is periodic in } \xi,\sigma.
\end{aligned}
\right.
\end{equation*}
From this and \cite{inthomo}, we can conclude that $\bar{F}^1=\bar{F}^0$.
\end{proof}
With this lemma, we derive our second layer cell problem:
\begin{equation*}
\text{({R2})}\quad\quad \left\{
\begin{aligned}&
\frac{\partial}{\partial t}w^z- \bar{F}^0(D^2w^z )=0 &  \text{ in }&\quad  P_{\nu},\\
&w^z(x,t)=\chi^{z}_{\nu}(ct) &  \text{ on }&\quad  \partial P_{\nu}.
\end{aligned}
\right.
\end{equation*}

As before, we study the boundary layer limit of $w^z$.

\begin{lemma}\label{R1R2}
Consider problem (R2) with rational $\nu$ and let $\psi^z$ be the unique solution of (R2). Then ${\lim}_{R\rightarrow \infty}w^z(x+R\nu,t)$ $(=:\psi(F^0,g^0,\nu)=\psi)$ exists and the limit
is independent of $x,t,z$ and $|c|$. Furthermore for some constant $N$ only depending on $\nu$ and for all $x,t\in P_\nu$, we have
\begin{equation}
    \label{r2 rt est}
    |w^{z}(x+R\nu,t)-\psi|\leq C\left({N}/{R}\right)^\beta.
\end{equation}
\end{lemma}

\begin{proof}
By the same proof of Lemma \ref{bd lyer lmt R1},
${\lim}_{R\rightarrow \infty}w^z(x+R\nu,t)$
exists and is independent of $x,t$.
Again by Lemma \ref{bd lyer lmt R1}, in terms of $z,\nu,t$, $\chi_\nu^z(ct)$ only depends on $z\cdot\nu+ct$. Therefore $\chi_\nu^z(ct)=\chi_\nu^0\left(ct+{z\cdot\nu}\right).$ From the equation, \[\tilde{w}(x,t):=w^0\left(x,t+\frac{z\cdot\nu}{c}\right)\] 
solves (R2) and by uniqueness $\tilde{w}(x,t)=w^z(x,t)=w^0\left(x,t+\frac{z\cdot\nu}{c}\right)$ (here we need $c\neq 0$). While the boundary layer limits of $w^0\left(x,t+\frac{z\cdot\nu}{c}\right)$ and $w^0(x,t)$ coincide which shows that the limit $\psi$ is independent of $z$.

To show the independence on $|c|$, consider $w'(x,t):=w^z(\frac{x}{\sqrt{|c|}},\frac{t}{|c|})$. Since $\bar{F}^0$ satisfies condition (G4), it can be checked that $w'$ solves (R2) with $c$ replace by $\frac{c}{|c|}$. By definition, the boundary layer limits of $w',w^z$ agree and hence $\psi$ is independent of $|c|$.
Finally \eqref{r2 rt est} follows the same as the proof of Lemma \ref{bd lyer lmt R1}.

\end{proof}

\begin{remark}\label{def psinu}
%\begin{flushleft}
From the two cell problems (R1)(R2), we can also write $\psi=\psi\left(F^{x,t},g^{x,t},\nu_{x,t},\frac{c_{x,t}}{|c_{x,t}|}\right)$ for $(x,t)\in \Gamma_2$.
Due to the lemma, there is no loss of assuming $z=\tau=0$ in the two cell problems. We might drop them from the notations of $w^z, v_s^{z,\tau}$. 

Let us also mention that the double homogenization procedure used for points in $\Gamma_2$ also works for the irrational case. Indeed if $(x_0,t_0)\in\Gamma_1$, $\chi^{z}_{\nu}(c\,\cdot)$ is just a constant and in this situation $\psi=\chi^z_\nu$.
%Lastly, the rate of convergence is not carefully studied here. We think similar results as in \cite{Kimcont} can be achieved.
%\end{flushleft}
\end{remark}

At the end of the section, we consider convex and translation invariant operators.  
\begin{lemma}\label{concave}
Suppose $F=F(M)$ is a uniformly elliptic, convex and homogeneous operator, $\nu$ is any unit vector. Function $g(x,t)$ is H\"{o}lder continuous and $\mathbb{Z}^d/\mathbb{Z}$ periodic in $x/t$. Let $u$ be a solution to
\begin{equation*}
\left\{
\begin{aligned}
&\frac{\partial}{\partial t}u- {F}(D^2 u )=0 &  \text{ in }&\quad  P_{\nu},\\
&u(x,t)=g(x,t) &  \text{ on }&\quad  \partial P_{\nu}.
\end{aligned}
\right.
\end{equation*}
Then $\bar{g}:=\lim_{R\rightarrow \infty}u(R\nu,0)$ exists and 
\[\bar{g}\geq \int_{[0,1)^{n+1}} g(x,t)dxdt.\]
The opposite inequality holds in case when $F$ is concave. In particular if $F(M)$ is linear, we have
\[\bar{g}= \int_{[0,1)^{n+1}} g(x,t)dxdt.\]
\end{lemma}
The proof follows from the one of Lemma 3.6 \cite{wf} and it is given by Riesz Representation Theorem.

\begin{corollary}
Suppose the homogenized operator $\bar{F}^0(M)$ in (R2) is linear. Then the boundary layer limit of (R2) satisfies: \[\psi=|c\hat{\nu}|\int_{t=0}^{|c\hat{\nu}|^{-1}}\chi_\nu^z(ct)dt=|\hat{\nu}|\int_{t=0}^{|\hat{\nu}|^{-1}}\chi_\nu^z(t)dt.\]
\end{corollary}

Later we show in section \ref{discont} that in general if the operator is not linear, the corresponding boundary layer limit might not be the same as the average of the original boundary data.

\section{Continuity of the Homogenized Boundary Data}\label{seccont}

In this section we mainly prove the following continuity property.
\begin{theorem}\label{tcont}
Let $\Gamma_T:=\Gamma_1(T)\cup\Gamma_2(T)\subset \partial_l\Omega$ and define
\begin{equation}\label{def bar g}
    \bar{g}(x,t)=\left\{\begin{aligned}
&    \varphi(F^{x,t},g^{x,t},\nu_{x,t})& \quad \text{ if }(x,t)\in \Gamma_1(T)\quad \text{( see Lemma \ref{irgbar} ) },\\
&    \psi\left(F^{x,t},g^{x,t},\nu_{x,t},\frac{c_{x,t}}{|c_{x,t}|}\right)& \quad \text{ if }(x,t)\in \Gamma_2(T)\quad \text{( see Remark \ref{def psinu} )}.
    \end{aligned}\right.
\end{equation}
Then $\bar{g}(x,t)$ is continuous on $\Gamma_T$.
\end{theorem}
Before the proof of the theorem, we need two lemmas. The first one shows the continuity of the boundary layer limit on the direction $\nu$ when $\nu$ is irrational. It is basically the parabolic generalization of Lemma 3.4 \cite{wf}. For the completion, we will provide the proof. 

\begin{lemma}\label{lcont1}
Suppose $F,g$ satisfy condition (G), $\nu_i\in S^{d-1}$ for $i=0,1$ and $\nu_0$ is irrational. Suppose for $i=0,1$, $v^{i}$ solves
\[
\left\{\begin{aligned}&
\frac{\partial}{\partial t}v^{i}(x,t)- F(D^2v^{i},x,t)=0 &  \text{ in }&\quad   P_{\nu_i},\\
&v^{i}(x,t)=g(x,t)&  \text{ on }&\quad  \partial P_{\nu_i}.
\end{aligned}
\right.
\]
Then there exists $C(\nu_0)$ such that for any $R>1$, if $|\nu_0-\nu_1|$ is sufficiently small, we have
\[\sup_{x,t\in P_{\nu_1}}|\varphi(F,g,\nu_0)-v^{1}(x+R\nu_1,t)|\leq C\left(R^{\frac{3}{2}\beta}|\nu_0-\nu_1|^\beta+R^{-\frac{1}{2}\beta}+w_{\nu_0}^\beta(R^\frac{1}{2})\right),\]
where $\varphi(F,g,\nu_0)$ is the boundary layer limit of $v^0$, $w_{\nu_0}$ is given in Lemma \ref{lemm discpcy} and $\beta$ in Theorem \ref{LW0}. 
Furthermore, if $\nu_1$ is also irrational, as a corollary we have for any $\delta>0$, there exists $C=C(\delta,\nu_0)$ such that
\[|\varphi(F,g,\nu_0)-\varphi(F,g,\nu_1))|\leq C|\nu_0-\nu_1|^\beta+\delta.\]
\end{lemma}

\begin{proof}
By Lemma \ref{irgbar}, for any $N>0$ there exists a constant $C$ independent of $ N,R$ such that
\[\sup_{x,t\in P_{\nu_0}}|v^{0}(x+R\nu,t)-\bar{g}(G,h,\nu_0)|\leq C\left(({N}/{R})^\beta+w_{\nu_0}(N)^\beta\right).\]

Now we compare $\varphi(F,g,\nu_0)$ with $v^{1}(x_1+R\nu,t_1)$ for any fixed $(x_1,t_1)\in\partial P_{\nu_1}$. Let $\tilde{v}^{i}(x,t):=v^{i}(x_1+x,t_1+t)$ which then solves
\begin{equation*}
\left\{\begin{aligned}&
\frac{\partial}{\partial t}\tilde{v}^{i}(x,t)- F(D^2
\tilde{v}^{i},x+x_1,t+t_1)=0 &  \text{ in }&\quad   P_{\nu_i},\\
&\tilde{v}^{i}(x,t)=g(x+x_1,t+t_1)&  \text{ on }&\quad  \partial P_{\nu_i}.
\end{aligned}
\right.
\end{equation*}
Again by Lemma \ref{irgbar}, since for the irrational case the boundary layer limit is independent of $z,\tau$, we have \[\varphi(F,g,\nu_0)=\varphi(F(,\cdot+x_1,\cdot+t_1),g(,\cdot+x_1,\cdot+t_1),\nu_0).\]
By shifting the operator and the boundary data, we can assume $x_1=0,t_1=0$.

%\begin{figure}\caption{Continuity of $\varphi$ on $\nu$}\label{fig 1}\centering\includegraphics[scale=0.4]{homo2}\end{figure}

For each $i=0,1$ and any $L>R\geq 1$, denote
\[Q^i_L:=\{(x,t)\,|\,0\leq x_{\nu_i}\leq L,-L^2\leq  t\leq L^2, |x_i'|\leq L\}\]
where $x_{\nu_i}=x\cdot{\nu_i}, x_i'=x-x_{\nu_i}\nu_i$, and \[\tilde{\nu}=\frac{\nu_0+\nu_1}{|\nu_0+\nu_1|},\, R'=\frac{|\nu_0+\nu_1|}{2}R,\, L'=0.9 L,\,\delta=L|\nu_0-\nu_1|,\,\tilde{Q}_{L'}=\overline{Q^{\tilde{\nu}}_{ L'}}+\delta\tilde{\nu}.\]
If $\delta$ is small enough, we can have
$\tilde{Q}_{L'} \subset Q^0_L \cap Q^1_L$. 
By Theorem \ref{cont}, H\"{o}lder continuity of solutions,
\[|v^{0}(x,t)-v^{1}(x,t)|\leq C\delta^\beta \text{ for } (x,t)\in \tilde{Q}_{L'}\cap\{x'\,|\,x'\cdot\tilde{\nu}=\delta\}.\]
Now we apply Lemma \ref{lemmalpos} in the region $\tilde{Q}_{L'}$ to obtain %Let $l_i\in \partial P_{\tilde{\nu}}$ be such that $R'\tilde{\nu}+l_i=R\nu_i$.
\begin{align*}
    |v^{0}(R\nu_0,0)-v^{1}(R\nu_1,0)|&=|v^{0}(R\nu_1,0)-v^{1}(R\nu_1,0)|+|v^{0}(R\nu_0,0)-v^{0}(R\nu_1,0)|\\
    &\leq C\delta^\beta+C\frac{R'}{L'}+C|R{\nu_0}-R{\nu_1}|^\beta\quad\text{ ( by H\"{o}lder continuity of $v^0$ )}\\
    &\leq C L^\beta|\nu_0-\nu_1|^\beta+C{R}/L.
\end{align*}
Next select $L=NR, N=R^\frac{1}{2}$ and then by \eqref{rate cvg irr}, we obtain
\begin{align*}
    |\varphi(F,g,\nu_0)-v^{1}(R\nu_1,0)|&\leq |\varphi(F,g,\nu_0)-v^{0}(R\nu_1,0)|+|v^{1}(R\nu_0,0)-v^{1}(R\nu_1,0)|\\
    &\leq C({N}/{R})^\beta+Cw_{\nu_0}(N)^\beta+C(RN)^\beta|\nu_0-\nu_1|^\beta+C/N\\
    &\leq C\left(R^{\frac{3}{2}\beta}|\nu_0-\nu_1|^\beta+R^{-\frac{1}{2}\beta}+w_{\nu_0}(R^{\frac{1}{2}})^\beta\right).
\end{align*}
We proved the first claim.

If $\nu_1$ is irrational, we have 
\[|v^{1}(R\nu_1,0)-\varphi(F,g,\nu_1)|\leq C\left(({N}/{R})^\beta+w_{\nu_1}^\beta(N)\right).\]
Lemmas 2.2, 3.4 in \cite{wf} show that if $|\nu_1-\nu_0|$ is small enough, $w_{\nu_1}(N)\leq 2 w_{\nu_0}(N)$. Thus we find 
\[|\varphi(F,g,\nu_0)-\varphi(F,g,\nu_1)|\leq C\left(R^{\frac{3}{2}\beta}|\nu_0-\nu_1|^\beta+R^{-\frac{1}{2}\beta}+w_{\nu_0}^\beta(R^\frac{1}{2})\right)\]
holds when $|\nu_0-\nu_1|,\frac{1}{R}$ are small enough. The second claim follows if we further take $R=R(\delta,\nu_0)$ to be large.
\end{proof}

In the following lemma, we prove the stability of cell problems on the operator and the boundary data.

\begin{lemma}\label{lcont2}
Let $\nu\in\mathbb{S}^{d-1}$. Assume $\{F^n(M,x,t),g^n(x,t), n\geq 0\}$ satisfy (G) and $F^n\rightarrow F^0$, $g^n\rightarrow g^0$ locally uniformly as $n\to\infty$. Suppose for each $n\geq 0$, $v^n$ is the unique bounded solution to
\begin{equation}\label{stability n}
\left\{\begin{aligned}&
\frac{\partial}{\partial t}v^n(x,t)- F^n(D^2v^n,x,t)=0 & \text{ in }&\quad   P_{\nu},\\
&v^n(x,t)=g^n(x,t)& \text{ on } &\quad  \partial P_{\nu}.
\end{aligned}
\right.
\end{equation}
Then $v^n\rightarrow v^0$ locally uniformly. 
\end{lemma}
\begin{proof}
This lemma is the parabolic generalization of Lemma 3.3 (i)(ii)(v) in \cite{wf}. To show the locally uniformly convergence, we take both upper and lower half-relaxed limits of $v^n$ which are denoted as $v^*,v_*$ respectively. Then $v^*\geq v_*$ and $v^*=v_*=g^0$ on the boundary. By Lemma \ref{lem stability}, $v^*,v_*$ are respectively sub and supersolution to \eqref{stability n} with $n=0$. Therefore $v^*\leq v_*$ by comparison which, combining with the definition, leads to $v^*=v_*=v^0$. The last equality holds because of the uniqueness of bounded solutions of \eqref{stability n}. Thus we proved the convergence.
\end{proof}

Now we are ready to prove the main theorem of this section.

\begin{proof}(of Theorem \ref{tcont}.)
First let $(x_0,t_0)\in\Gamma_1$. Take any sequence $\Gamma_T\ni(x_n,t_n)\rightarrow(x_0,t_0)$ as $n\to\infty$. Because by definition $\Gamma_2$ is open, we can assume $(x_n,t_n)\in\Gamma_1$. For simplicity denote $F^n=F^{x_n,t_n},g^n=g^{x_n,t_n}$ and we know $F^n\rightarrow F^0,g^n\rightarrow g^0$ locally uniformly.
By Lemma \ref{lcont1} and Lemma \ref{lcont2},
\begin{align*}
&    \lim_{n\rightarrow \infty}|\bar{g}(x_n,t_n)-\bar{g}(x_0,t_0)|\\
=&    \lim_{n\rightarrow \infty}|\varphi(F^n,g^n,\nu_n)-\varphi(F^0,g^0,\nu_0)|\\
\leq&
\lim_{n\rightarrow \infty}|\varphi(F^n,g^n,\nu_n)-\varphi(F^n,g^n,\nu_0)|+\lim_{n\rightarrow \infty}|\varphi(F^n,g^n,\nu_0)-\varphi(F^0,g^0,\nu_0)|\\
=\,&0,
\end{align*}
which provides the continuity of $\bar{g}$ on $\Gamma_1$.

Next fix $(x_0,t_0)\in\Gamma_2$. By definition, the spatial normal vector $\nu$ is constant in a small neighbourhood. Take $(x_n,t_n)\in\Gamma_2$ inside the neighbourhood and suppose $(x_n,t_n)$ converges to $ (x_0,t_0)$. Write the boundary speed at $t_n$ as $c_n$ and then $c_n\rightarrow c_0\ne 0$. And we can assume that $c_n,c_0$ are of the same sign. Suppose for each $n$, $v^{n}_s$ solves (R1) at point $(x_n,t_n)$ with $z=\tau=0$. Let $\chi^{n}(c_n s)$ be the boundary layer limit of $v^{n}_s$. By Lemma \ref{bd lyer lmt R1}, $\{\chi^n(\cdot)\}_{n\geq 0}$ are periodic with periodicity $\frac{1}{|\hat{\nu}|}$. Then Lemma \ref{lcont2} implies that for all $t$,
\[
\chi^{n}(c_0t)\rightarrow \chi^{0}(c_0t).
\]

Let $\bar{F}^n$ be the homogenized operator of $F^n$ for all $n\geq 0$. By Proposition 3.2 \cite{intho} and locally uniform convergence of $F^n$, we deduce that $ \bar{F}^n\to \bar{F}^0$ locally uniformly.
Now let $w^{n}$ be the solution to (R2) with boundary data $\chi^n(c_0 \,\cdot)$. From above, both the operators and the boundary data converge. So the solutions $\{w^{n}\}$ converge locally uniformly and the limit equals $w^0$ which is the unique solution to (R2) with operator $\bar{F}^0$ and boundary data $\chi^0(c_0\,\cdot)$. 

According to Lemma \ref{R1R2}, since $c_n,c_0$ are of the same sign, the boundary layer limit of $w^n$ equals $\psi^n$. Furthermore, for some universal $C$ and a constant $N(\nu)$ we have for each $n$,
\[|w^n(R\nu+x,t)-\psi^n|\leq C\left({N}/{R}\right)^\beta\]
where $\psi^n=\psi^n(F^n,g^n,\nu,c_n)$ is the boundary layer limit of $w^n$.
Since the convergence is uniform in $n$ and $w^n\to w^0$ locally uniformly, we have
\[\lim_{n\rightarrow \infty}|\bar{g}(x_n,t_n)-\bar{g}(x_0,t_0)|=\lim_{n\rightarrow \infty}|\psi^n-\psi^0|=0.\]
We proved that $\bar{g}$ is continuous on $\Gamma_2$.
\end{proof}
%As a remark, if $\nu$ is not fixed, it does not make sense to define $\bar{g}$ in the above way since it may not be continuous even under all other assumptions.

\subsection{Continuity Extension for Special Operators}\label{CELO}

In this section we study the continuity extension of the homogenized boundary data which corresponds to the study of the continuous dependence of the boundary layer limits on operators and boundary data, and the comparison of the two types of boundary layer tail problems. Because of the failure of the continuity extension in general, we will put more assumptions on operators.

Following from the idea of \cite{kim,Kimcont}, the boundary layer limit $\varphi(\nu)=\varphi(F^0,g^0,\nu)$ of problem (I) has a directional limit as $\nu$ approaches one rational direction $\xi$ and the limit can be characterized by a second boundary layer problem. We state the theorem below. 

\begin{theorem}\label{5.1}
For any fixed rational unit vector $\xi$ and a unit vector $\eta$ perpendicular to $\xi$. Let $\zeta:[0,1)\rightarrow \mathbb{S}^{d-1}$ be a geodesic path with unit
speed and $\zeta(0)=\xi,\dot{\zeta}(0)=\eta$. Recall equation (I) with irrational normal $\zeta(s)$, let $\varphi(F^0,g^0,{\zeta(s)})$ be the boundary layer limits. Let $M_\zeta(s)$ be the boundary layer limit of the following problem (which is (R1) with $c=1$)
\begin{equation} \label{limit0}
\left\{
\begin{aligned}&
{\partial_t}v_{s}(x,t)- F^0(D^2v_{s},x+{s \xi},t)=0 &  \text{ in }& \quad  P_\xi,\\
&v_{s}(x)=g^0(x+ s\xi,t ) &  \text{ on }&\quad   \partial P_\xi.
\end{aligned}
\right.
\end{equation}
Let $L_\xi=L_\xi(F^0,g^0,\eta)$ be the boundary layer limit of  
\begin{equation} \label{limit}
\left\{
 \begin{aligned}
&{\partial_t}w_{\xi,\eta}-\bar{F}^0(D^2w_{\xi,\eta})=0  & \text{ in }&\quad  P_\xi,\\
&w_{\xi,\eta}(x,t)=M_\xi(x\cdot\eta)  & \text{ on } &\quad  \partial P_\xi.
 \end{aligned}
 \right.
\end{equation}
Then $L_\xi(F^0,g^0,\eta)=\lim_{s\to 0}\varphi(F^0,g^0,\zeta(s))$.

\end{theorem}

For the proof, we refer readers \cite{kim,Kimcont,feldman2018continuity}.
The theorem is significant since it gives a characterization of the directional limit of $\varphi(\nu)$ at rational direction. With this, we give our main theorem of this section.

\begin{theorem} \label{homoF}
(Continuity extension) 
\begin{enumerate}
    \item [(i.) ] Let $(x_0,t_0)\in \overline{\Gamma}_1$ and write the corresponding rational normal vector as $\xi$. Suppose $\bar{F}^{x_0,t_0}$ is rotation/reflection invariant. Then $L_\xi(F^{x_0,t_0},g^{x_0,t_0},\eta)$ is independent of $\eta$ and for any $(x_n,t_n)\in\Gamma_1$ such that $(x_n,t_n)\to (x_0,t_0)$ as $n\to\infty$, we have $\bar{g}(x_n,t_n)\to L_\xi(F^{x_0,t_0},g^{x_0,t_0},\eta)$.
    
    \item [(ii.) ] Let $(x_0,t_0)\in \overline{\Gamma}_1\cup \overline{\Gamma}_2$ and write the corresponding rational normal as $\xi$. Suppose $\bar{F}^{x_0,t_0}$ is linear. Then $L_\xi(F^{x_0,t_0},g^{x_0,t_0},\eta)$ is independent of $\eta$ and for any $(x_n,t_n)\in\Gamma_1\cup\Gamma_2$ such that $(x_n,t_n)\to (x_0,t_0)$ as $n\to\infty$, we have $\bar{g}(x_n,t_n)\to L_\xi(F^{x_0,t_0},g^{x_0,t_0},\eta)$.
\end{enumerate}

\end{theorem}
\begin{proof}

First we show that $L_\xi(F^{x_0,t_0},g^{x_0,t_0},\eta)$ is independent of $\eta$. Because $\bar{F}^0=\bar{F}^{x_0,t_0}$ is rotation/reflection invariant, the equation in \eqref{limit} is preserved by unitary transformations on $x$. Therefore for any $\eta_1\perp \xi,\eta_2\perp \xi$, by the uniqueness of \eqref{limit},
\[w_{\xi,\eta_1}(x,t)=w_{\xi,\eta_2}(Qx,t)\]
where $Q$ is an orthogonal matrix satisfying $Q\xi=\xi$ and $Q\eta_1=\eta_2$. By Lemma \ref{R1R2} the boundary layer limits of $w_{\xi,\eta_1}$ and $w_{\xi,\eta_2}$ are the same. Thus $L_\xi(F^{0},g^{0}):=L_\xi(F^{x_0,t_0},g^{x_0,t_0},\eta)$ is independent of $\eta$.

Next by definition, since for $n\geq 1$ the normal vector $\nu_n$ at $(x_n,t_n)$ is irrational, \[\bar{g}(x_n,t_n)=\varphi(F^{x_n,t_n},g^{x_n,t_n},{\nu_n}).\] By Lemma \ref{lem stability}, to show the existence of $\lim_{n\to\infty}\bar{g}(x_n,t_n)$ and identify the limit, we only need to show
\[\lim_{n\to\infty}\varphi(F^{x_0,t_0},g^{x_0,t_0},{\nu_n})=L_\xi(F^{0},g^{0})\]
which is certainly true by Theorem \ref{5.1}. We proved (i).

\smallskip

Write $\xi=\frac{\hat{\xi}}{|\hat{\xi}|}$ where $\hat{\xi}\in\mathbb{Z}^d$ is irreducible.
Notice that \eqref{limit0} is the same as (R1) with $c=1,z=0,\tau=0,\nu=\xi$. By Lemma \ref{bd lyer lmt R1}, $M_\xi(s)=\chi_{s}=\chi(s)(c=1)$ is periodic in $s$ with periodicity $|\hat{\xi}|^{-1}$.
When $\bar{F}^0$ is linear, by Lemma \ref{concave} we get for any $\eta$,
\[
L_\xi(F^0,g^0,\eta)=|\hat{\xi}|\int_0^{1/|\hat{\xi}|} M_\xi(t)dt.
\]
The right-hand side is independent of $\eta$ and so is $L_\xi(F^0,g^0,\eta)=:L_\xi(F^0,g^0)$.

If $(x_n,t_n)\in \Gamma_1$, as before by Theorem \ref{5.1} and stability of boundary layer limit,
$\bar{g}(x_n,t_n)\to L_\xi(F^0,g^0)$ as $n\to \infty$. If $(x_n,t_n)\in \Gamma_2$, let $c_0\neq 0$ to be the speed of the boundary at $x_0,t_0$. After comparing (R2) and \eqref{limit}, we find $M_\xi(\cdot)=\chi_\xi(\cdot)$. Thus by Lemma \ref{concave},
\begin{align*}
    \psi(F^0,g^0,\nu_0,c_0)&=|c_0\hat{\xi}|\int_0^{1/|c_0\hat{\xi}|} \chi_\xi(c_0t)dt\\
    &=|\hat{\xi}|\int_0^{1/|\hat{\xi}|} \chi_\xi(t)dt\\
    &=|\hat{\xi}|\int_0^{1/|\hat{\xi}|} M_\xi(t)dt=L_\xi(F^0,g^0).
\end{align*}
Finally by stability of the boundary layer limits, we have
$\bar{g}(x_n,t_n)\to L_\xi(F^0,g^0)$ as $n\to \infty$.

\end{proof}

\begin{remark}
\label{rmrk linear}

Theorem \ref{homoF} illustrates that in the case the homogenized operator $\bar{F}$ is linear, $\bar{g}$ can be extended continuously to $\overline{\Gamma}_T$: the entire lateral boundary except non-moving flat boundary parts of rational normal. This observation will lead to Theorem \ref{linearversion}.

When $F(M,x,y,t,s)$ is linear in $M$, in view of the proof of Lemma \ref{homobarF}, $\bar{F}(M,x,t)$ is linear in $M$. 

\end{remark}

\subsection{Discontinuity of Homogenized Boundary Data}\label{discont}

The discontinuity of the boundary layer tail is discussed in \cite{Kimcont} for elliptic problems.
In this section, we give two examples showing the failure of continuous extension of $\bar{g}$ if the parabolic operator $F$ is not linear.

%So it is possible that $\bar{g}$ is discontinuous at boundary points of rational normal with non-zero speed when the surrounding boundary is not locally spatial flat.

\begin{example}\label{prop4.6}
Recall the definition of $\bar{g}$ in Theorem \ref{tcont}. There exist $F,g$ satisfying (F1)--(F4), such that $\bar{g}$ cannot be continuously extended to $\overline{\Gamma}_1\cap \overline{\Gamma}_2$. The same result holds even if $F$ is rotation/reflection invariant.
\end{example}

\begin{proof}

The first example is given with non-rotation/reflection invariant operator. Because only local information of the original problem is needed to identify the homogenized boundary data, without loss of generality we can consider problems in unbounded domains. Consider the following continuous space-time region in $\mathbb{R}^3$,
\[\Omega(t):=\{(x,y)\,|\,y>t \text{ for }t<0\}\cup\{(x,t)\,|\,y-\tan(t)x>t \text{ for }t\geq 0\}.\]
Write $\Omega=\bigcup_{t\in (-1,1)}(\Omega(t)\times\{t\})$.

Denote two spacial vectors $(1,0),(0,1)$ as $e_1,e_2$ respectively.
Consider a set of spatial {orthogonal bases} which are close to $e_1,e_2$: for $\delta\in(0,1)$, \[\mathcal{B}_{\delta}=\{(\eta,\nu)\,|\,\eta=\cos(\tau)e_1+\sin(\tau)e_2, \nu=-\sin(\tau)e_1+\cos(\tau) e_2, -\delta<\tau<\delta\}.\]

Let $u^\eps$ be a solution to
\begin{equation}\label{example}
\left\{
\begin{aligned}
&\frac{\partial u^\eps}{\partial t}-\sup_{(\eta,\nu)\in \mathcal{B}_{\delta}}\{\Delta u^\eps,4u^\eps_{\eta\eta}+u^\eps_{\nu\nu}\}=0  &\text{ in }&\quad  \Omega\\
&u^\eps (x,y,t)=\sin(\frac{y}{\eps}) &
\text{ on }&\quad  \partial_l \Omega.
\end{aligned}
\quad \right.
\end{equation}
It is straightforward that
\begin{equation*}
    \begin{aligned}
& u^\eps(x,y,t)=\sin(\frac{t}{\eps})\text{ on }y=t, &\text{ for }t<0,\\
& u^\eps(x,y,t)=\sin(\frac{y}{\eps}) \text{ on } y-\tan(t)x=t, & \text{ for }t\geq 0. 
\end{aligned}
\end{equation*}

First let us compute the homogenized boundary data for $t<0$. Fix any $(x,y,t)=(x_0,t_0,t_0)$ with $t_0<0$. At this point, the inner normal vector is $e_2$ and the boundary speed is $1$. Then for the first cell problem (R1) (with parameter $z=\tau=0,s$), we use the operator $F$ in \eqref{example}, boundary data $\sin(s)$ and the domain is $P_{e_2}$. 
By uniqueness, the solution to (R1) is just a constant: $v_s=\sin s$ which shows that the boundary layer limit
\[\chi_{e_2,s}=\chi(F,\sin(\cdot),e_2,1)=\sin s.\] 
Since $F$ only depends on $D^2u$, we get  $\bar{F}=F$. Then the corresponding second boundary layer problem (R2) becomes
\begin{equation}\label{R2''}
\left\{
\begin{aligned}
&{\partial_t w}-\sup_{(\eta,\nu)\in \mathcal{B}_{\delta}}\{\Delta w,4w_{\eta\eta}+w_{\nu\nu}\}=0  &\text{ in }&\quad  P_{e_2}\\
&w (x,y,t)=\sin(t) &
\text{ on }&\quad  \partial P_{e_2}.
\end{aligned}
\quad \right.
\end{equation}
By uniqueness again, the solution to \eqref{R2''} is constant in $x$-direction. So we only need to solve for the following equation in $\mathbb{R}\times\mathbb{R}^+$:
\begin{equation*}
\left\{\begin{aligned}&
{\partial_t \tilde{w}}-\sup_{|\theta-1|\leq \delta'}\{\theta \tilde{w}_{yy}\}=0  &\text{ in }&\quad \{y>0\}\\
&\tilde{w}(y,t)=\sin t  &\text{ on } &\quad \{y=0\}
\end{aligned}
\right.
\end{equation*}
and we have $w(x,y,t)=\tilde{w}(y,t)$.
Here $\delta'$ depends on $\delta$ which can be arbitrarily small if $\delta$ is small.
Notice that if $\delta'$ is $0$, the operator is linear. For linear equation, Lemma \ref{concave} implies that the boundary layer limit $\psi(\delta=0)=\lim_{R\to\infty}\tilde{w}(R,0)$ is the average of $\sin t$ which is $0$.
Then by the stability of viscosity solutions, we find that $\bar{g}(x_0,y_0,t_0)$ can be arbitrarily close to $0$ which is uniform in $(x_0,y_0,t_0)$ if $\delta$ is small enough.  

Next we compute the boundary data for $(x,y,t)=(0,t_0,t_0)$ with $t_0>0$ small enough. Denote
\[z_1=\cos( t_0) x+\sin(t_0)y, \quad z_2=-\sin(t_0)x+\cos(t_0)y.
\]
So 
$x=\cos(t_0)z_1-\sin(t_0)z_2$, $ y=\sin(t_0)z_1+\cos(t_0)z_2.$
We change the coordinate from $(x,y)$ to $(z_1,z_2)$. If $(\cos t_0,\sin t_0)$ is irrational, we need to use cell problem (I) which is
\begin{equation}\label{homoeta}
\left\{\begin{aligned}&
\partial_t v-\sup_{(\eta,\nu)\in \mathcal{B}_{\delta}}\{\Delta v,4v_{\eta\eta}+v_{\nu\nu}\}=0  &\text{ in }&\quad \{z_2>0\},\\
&v (z_1,z_2,t)=\sin\left(\sin(t_0)z_1\right) &
\text{ on }&\quad \{z_2=0\}.
\end{aligned}
\right.
\end{equation}
If $t_0$ is small enough,
$
z_1,z_2$ directions are belong to $ \mathcal{B}_{\delta}$. Then it is not hard to check that
\[v_1=e^{-\sin(t_0)2z_2}\sin\left(\sin(t_0)z_1\right),\quad v_2=e^{-\sin(t_0)z_2}\sin\left(\sin(t_0)z_1\right)\]
are two subsolutions to \eqref{homoeta}. Therefore by comparison, 
\[v(z_1,z_2,t)\geq \max\{v_1(z_1,z_2,t),v_2(z_1,z_2,t)\}=:\kappa(z_1,z_2).\] 
Direct computation shows that on the hyperplane $\{(z_1,z_2)\,|\,z_2=(\sin(t_0))^{-1}\}$, we have
\[\kappa=e^{-2}[\sin\left(\sin(t_0)z_1\right)]_-+e^{-1}[\sin\left(\sin(t_0)z_1\right)]_+\]
Also since the operator $F$ is convex, by Lemma \ref{concave}, we have
\[\bar{g}(0,t_0,t_0)=\lim_{R\rightarrow \infty}v(0,R,0)\geq \frac{\sin(t_0)}{{2\pi}}\int_0^\frac{2\pi}{\sin(t_0)}\kappa(z_1)dz_1=:\alpha>0\]
which is an universal constant independent of $t_0,\delta$.

Now we fix to $\delta$ small enough such that $\bar{g}(x,y,t)<\frac{c}{2}$ for all $t<0$. Then take $\alpha_0>0$ small enough such that for all $0<t<\alpha_0$, we have $\bar{g}(0,t,t)\geq \alpha$. This shows that $\bar{g}(x,y,t)$ is not continuous at point $(0,0,0)\in \overline{\Gamma}_1\cap \overline{\Gamma}_2$.

\smallskip

Next we construct an example with rotation/reflection invariant operator and we will show that the continuous extension of $\bar{g}$ on $\overline{\Gamma}_T$ still fails. For simplicity, we only work on the cell problems and we are going to show that the boundary layer limit $\psi_\xi$ from (R1)(R2) can be different from $L_\xi$ from \eqref{limit0} and \eqref{limit}.

Let \[P_\xi:=\{(x,y,t)\in\mathbb{R}\times \mathbb{R}^+\times\mathbb{R}\}\] and
$\xi$ is to positive $y$-direction. Let $F(M)=\max\left\{tr M, \Lambda tr M\right\}$ with $\Lambda>1$ which is a rotation/reflection invariant operator. We select $g(x,y)=\sin y$.
Since $c=1$, (R1) and \eqref{limit0} on $P_\xi$ coincide and notice that the solutions are just constants. We get
\[\chi(F,g,\xi,0)(s)=M_\xi(s)=\sin s.\]
Let $\eta$ in \eqref{limit} be $x$-direction and then the equation becomes
\begin{equation*}
\left\{\begin{aligned}&
{\partial_t}w_{\xi,\eta}-\max\left\{\partial_x^2{w_{\xi,\eta}}+\partial_y^2{w_{\xi,\eta}},\Lambda(\partial_x^2{w_{\xi,\eta}}+\partial_y^2{w_{\xi,\eta}})\right\}=0   & \text{ for } &\quad  y> 0,\\
&{w_{\xi,\eta}}(x,y,t)=\sin x  & \text{ for }&\quad  y=0.
\end{aligned}
\right.
\end{equation*}
By uniqueness of the equation, $w(x,y,t)=\sin(x)e^{-y}$ is the bounded solution which tells that the boundary layer limit \[L_\xi(F,\sin(\cdot))=\lim_{y\rightarrow\infty}\sin(x)e^{-y}=0.\]

As for (R2), we have
\begin{equation*}
\left\{\begin{aligned}&
\frac{\partial}{\partial t}w- \max\left\{w_{xx}+w_{yy},\Lambda(w_{xx}+w_{yy})\right\}=0  & \text{ for }&\quad  y> 0,\\
&w(x,y,t)=\sin t &\text{ for }&\quad  y=0.
\end{aligned}
\right.
\end{equation*}
Notice that 
\[\sin(t-\frac{1}{\sqrt{2\Lambda}}y)e^{-\frac{1}{\sqrt{2\Lambda}}y},\quad \sin(t-\frac{1}{\sqrt{2}}y)e^{-\frac{1}{{2}}y}\]
are two subsolutions with the same boundary data $\sin t$. Therefore by comparison, 
\[w(x,y,t)=w(y,t)\geq \max\left\{\sin(t-\frac{1}{\sqrt{2\Lambda}}y)e^{-\frac{1}{\sqrt{2\Lambda}}y}, \sin(t-\frac{1}{\sqrt{2}}y)e^{-\frac{1}{\sqrt{2}}y}\right\}.\]
We denote the right-hand side in the above by $\kappa'(y,t)$. Since the operator is convex, by Lemma \ref{concave},
\[\bar{g}=\psi(F,f,\xi,1)\geq \frac{1}{2\pi}\int_0^{2\pi} \kappa'(1,t)dt>0=L_\xi(F,f)\]
which finishes the proof.

%This finishes the proof.% that the limits from the two double-scale homogenization cell problems maybe different. From this we conclude: it is possible that the homogenized boundary data $\bar{g}$ cannot be defined continuously at points $\in\overline{\Gamma}_1\cap\overline{\Gamma}_2$ even for rotation/reflection invariant operators.
\end{proof}

\section{Homogenization on the Lateral Boundary}\label{sechomo}

In this section we go back to the original problem \eqref{eqns1} and prove the homogenization on $\Gamma_T$.

\begin{theorem}\label{homol}
Suppose conditions (F1)--(F4)(O) are satisfied. Let $(x_0,t_0)\in\Gamma_T$ and take any $(x_\eps,t_\eps)\in\overline{\Omega}_T$ such that $(x_\eps,t_\eps)\rightarrow (x_0,t_0)$ as $\eps\to 0$. Let $\bar{g}$ be the function on $\Gamma_T$ given in \eqref{def bar g}. If $(x_0,t_0)\in \Gamma_1$, then for any $\delta>0$, there exists $R_0$ such that for all $R>R_0$ we have 
\begin{equation}
    \label{r bndry cnvrg}
    \limsup_{\eps\rightarrow 0}|u^\eps(x_\eps+\eps R\nu,t_\eps)-\bar{g}(x_0,t_0)|\leq \delta.
\end{equation}
If $(x_0,t_0)\in \Gamma_2$, for any $\delta>0$ there exists $R_0'$ such that for all $R>R_0'$ we have
\begin{equation}
    \label{irr bndry cnvrg}
    \limsup_{\eps\rightarrow 0}|u^\eps(x_\eps+\eps^{\frac{1}{2}} R\nu,t_\eps)-\bar{g}(x_0,t_0)|\leq \delta.
\end{equation}
\end{theorem}

The key to proving the theorem is a local uniform convergence result (see \eqref{e0r0}, \eqref{claimunif}).
Let us start with the case when $\{(x_\eps,t_\eps)\}$ are on the boundary.

\begin{lemma}\label{contu}
Let $u^\eps$ be the solution to equation \eqref{eqns1} and take any sequence $(x_\eps, t_\eps)\in \partial_l\Omega_T$ such that $(x_\eps, t_\eps)\to (x_0,t_0)\in \Gamma_T$. Suppose along a sequence of $\eps \rightarrow 0$, $\eps^{-1}{x_\eps}\rightarrow z, \eps^{-2}{t_\eps}\rightarrow \tau $ in $\mathbb{R}^d/\mathbb{Z}^d$ and $\mathbb{R}/\mathbb{Z}$ respectively. Then if letting $v^\eps(x,t):=u^\eps(x_\eps+\eps x,t_\eps+\eps^2 t)$, $v^\eps(x,t)$ converges locally uniformly, along the sequence of $\eps\to 0$, to $v(x,t)$ which is the unique solution to
\begin{equation}\label{I}
\left\{\begin{aligned}&
\frac{\partial}{\partial t}v(x,t)- F^0(D^2v,x+z,t+\tau)=0 &  \text{ in }&\quad   P_{\nu},\\
&v(x,t)=g^0(x+z,t+\tau)&  \text{ on }&\quad   \partial P_{\nu}.
\end{aligned}
\right.
\end{equation}
\end{lemma}
\begin{proof}
Without loss of generality, assume $z=\tau=0$ and we write $\eps\to 0$ meaning along the sequence of $\eps \to 0$. By \eqref{eqns1}, $v^\eps$ satisfies
\[
\left\{
 \begin{aligned}
&\frac{\partial}{\partial t}v^\eps(x,t)-\eps^2 F(\eps^{-2}D^2v^\eps,x_\eps+{\eps}{x},\eps^{-1}x_\eps+x,t_\eps+\eps^2 t,\eps^{-2}t_\eps+t)=0    &\text{in }&  \Omega^\eps,\\
&v^\eps(x,t)=g(x_\eps+\eps{x},\eps^{-1}x_\eps+x,t_\eps+\eps^2 t,\eps^{-2}t_\eps+t)   &\text{on }&  \partial_l\Omega^\eps .
 \end{aligned}
 \right.
\]
The domains converge locally uniformly to $P_\nu$ in Hausdorff distance as $\eps\to 0$.
The operators 
\[\eps^2 F(\eps^{-2}M,x_\eps+{\eps}{x},\eps^{-1}x_\eps+x,t_\eps+\eps^2 t,\eps^{-2}t_\eps+t)\] 
converge locally uniformly to $F^0(M,x,t)$ by condition (F1)(F4).

Let $v^*,v_*$ be respectively the upper and lower half-relaxed limits of $v^\eps$ which are then functions defined in $P_{\nu}$. By stability of viscosity solutions (Lemma 2.4 in \cite{wf}), $v^*$ and $v_*$ are respectively sub and supersolutions of equation \eqref{I}. If $v^*=v_*$ on $\partial P_\nu$, by Corollary \ref{cols}, $v^*\leq v_*$. However since the inverse inequality holds by definition, we obtain $v^*=v_*$ in $P_\nu$ which gives the desired convergence result. Consequently we are left to show the convergence of $v^\eps$ on the boundary.

Fix any point $(y_0,s_0)\in\partial{P}_\nu$. Take any $(y_\eps,s_\eps)\rightarrow (y_0,s_0)$ such that $(x_\eps+\eps y_\eps,t_\eps+\eps^2 s_\eps)\in\overline{\Omega}_T$. It follows from H\"{o}lder continuity of both $v^\eps,g$, and the assumption $\eps^{-1}{x_\eps}\rightarrow 0$,  $\eps^{-2}{t_\eps}\rightarrow 0 $ that
\begin{align*}
    &\lim_{\eps\rightarrow 0}|v^\eps(y_\eps,s_\eps)-g^0(z+y_0,\tau+s_0)|\\
    =\,&\lim_{\eps\rightarrow 0}|v^\eps(y_\eps,s_\eps)-g(x_\eps+\eps y_\eps,\eps^{-1}x_\eps+y_\eps,t_\eps+\eps^2s_\eps,\eps^{-2}t_\eps+s_\eps)|\\
    =\,&0.
\end{align*}
This shows that $v^*=v_*$ on the boundary which finishes the proof. %Then by comparison and definitions we have $v^*=v_*$ in $P_\nu$, following from which we claim we proved this lemma.

\end{proof}

Next we study the case when $(x_0,t_0)\in\Gamma_2(T)$.

\begin{proposition}\label{contpsi}
Let $u^\eps$ be the solution to equation \eqref{eqns1} and take any sequence $(x_\eps, t_\eps)\in \partial_l\Omega_T$ such that $(x_\eps, t_\eps)\to (x_0,t_0)\in \Gamma_2$. Suppose along a sequence of $\eps \rightarrow 0$, $\eps^{-1}{x_\eps}\rightarrow z, \eps^{-2}{t_\eps}\rightarrow \tau $ in $\mathbb{R}^d/\mathbb{Z}^d$ and $\mathbb{R}/\mathbb{Z}$ respectively. Write $c_\eps=c(t_\eps)$ as the speed of the boundary at time $t_\eps$ (since $(x_0,t_0)\in \Gamma_2$, $c$ is locally independent of $x$). Then if letting  
\[w^\eps(x,t):=u^\eps(x_\eps+\eps^\frac{1}{2} x+c_\eps\eps t\nu,t_\eps+\eps t),\]  $w^\eps(x,t)$ converges locally uniformly along the sequence of $\eps\to 0$ to $w^z(x,t)$ which is the unique solution to (R2).
\end{proposition}

\begin{proof}
Without loss of generality, assume $z=\tau=0, c(t_0)=1$. 
Let $w^*,w_*$ be the upper and lower half-relaxed limits of $w^\eps$. Since $w^\eps$ is a solution to \eqref{derive R2}, by Lemma \ref{homobarF} and stability of viscosity solutions, $w^*$ and $w_*$ are respectively sub and super solutions to
\[
\partial_t{w}-\bar{F}(D^2 w)=0 \text{ in }  P_{\nu}.\]
As before, we only need to show that $w^\eps(x,t)$ converges to $\chi_\nu(t)$ on the boundary where $\chi_\nu(t)=\chi^0_\nu(t)$ is given in Lemma \ref{bd lyer lmt R1}.
The fact that $\nu$ is locally a constant vector will be essential.

Fix any point $(y_0,s_0)\in\partial{P}_\nu$. First we take $(y_\eps,s_\eps)\rightarrow (y_0,s_0)$ such that 
\begin{equation}
    \label{bd point}
    (x_\eps+\eps^\frac{1}{2}y_\eps+c_\eps\eps t\nu,t_\eps+\eps s_\eps)\in\partial_l{\Omega_T}.
\end{equation}
Let $\mu^\eps(x,t)=w^\eps(y_\eps+\eps^\frac{1}{2} x,s_\eps+\eps t)$. For abbreviation of notations, we write
$
\hat{x}_\eps=x_\eps+\eps^{\frac{1}{2}}y_\eps,\hat{t}_\eps=t_\eps+\eps s_\eps.$
Then $\mu^\eps$ satisfies
\begin{align*}
&\quad\quad\partial_t \mu^\eps-\eps  c_\eps\nu\cdot D\mu^\eps\\
&-\eps^2F(\eps^{-2}D^2\mu^\eps,\hat{x}_\eps+\eps x+c_\eps( \eps s_\eps+\eps^2t)\nu,
\eps^{-1}\hat{x}_\eps+ x+c_\eps( s_\eps+\eps t)\nu,\hat{t}_\eps+\eps^{2}t,\eps^{-2}\hat{t}_\eps+t)=0.
\end{align*}
The domain converges to a half plane in Hausdorff distance and for $(x,t)$ on the boundary and near the origin 
\begin{equation}\label{webdry}
\mu^\eps(x,t)=g(\hat{x}_\eps+\eps x+c_\eps( \eps s_\eps+\eps^2t)\nu,
\eps^{-1}\hat{x}_\eps+ x+c_\eps( s_\eps+\eps t)\nu,\hat{t}_\eps+\eps^{2}t,\eps^{-2}\hat{t}_\eps+t).    
\end{equation}
By passing to a subsequence again, we can assume
$\eps^{-1}\hat{x}_\eps\rightarrow z_0, \eps^{-2}\hat{\tau}_\eps\rightarrow \tau_0$. Therefore on $\partial P_\nu$, $\mu^\eps(x,t)$ converges locally uniformly to $g(x_0,z_0+x+s_0\nu,t_0,\tau_0+t)$. 
The key fact here is $z_0\cdot\nu=0$. Indeed by the fact that $\partial_l\Omega_T$ is $\mathcal{C}^1$ and is locally spatially flat near $x_0,y_0$, and by \eqref{bd point}, we have \[\eps^\frac{1}{2}|(y_\eps-y_0)\cdot\nu|=o(\eps|s_0|)\] and so \[\eps^{-1}\hat{x}_\eps\cdot\nu=\eps^{-1}x_\eps\cdot\nu+\eps^{-\frac{1}{2}} y_\eps\cdot\nu\to 0.\] We remark that the same estimate does not hold if $\nu$ varies in space direction. In that case we only have $\eps^\frac{1}{2}|y_\eps-y_0|\leq o(\eps^\frac{1}{2}|y_0|+\eps|s_0|)$. 

Then we send $\eps\to 0$ along the subsequence, and take the upper and lower half-relaxed limits of $w^\eps$, denoted as $\tilde{w}^*,\tilde{w}_*$ which are then respectively the sub and super solutions to
\begin{equation}\label{eqnwz0}
\left\{\begin{aligned}&
\frac{\partial w}{\partial{t}}-F^0(D^2w,z_0+x+ s_0\nu,\tau_0+t)=0 &\text{ in } &\quad  P_{\nu}\\
&w(x,t)=g^0(z_0+x+ s_0\nu,\tau_0+t)  &\text{ on } &\quad  \partial{P_{\nu}}.
\end{aligned}
\right.
\end{equation}
Due to \eqref{webdry}, $\tilde{w}^*=\tilde{w}_*$ on $\partial\mathcal{P}_\nu$. Then by comparison they are equal and we denote it by $w^{z_0}$ which solves \eqref{eqnwz0}. By Lemma \ref{bd lyer lmt R1}, the boundary layer limit $\lim_{R\rightarrow\infty}w^{z_0}(x+R\nu,t)$ exists and it only depends on $z_0\cdot\nu+s,\nu,F^0,g^0$. Since $z_0\cdot\nu=0$, we have
\[\lim_{R\rightarrow\infty}w^{z_0}(x+R\nu,t)=\chi_\nu(s_0).\]
Hence we proved that for any $\delta>0$, there exists $R_0$ such that for $R>R_0$
\[
    \limsup_{\eps\rightarrow 0}|w^\eps(y_\eps+\eps^\frac{1}{2}R\nu,s_\eps)-\chi_{\nu}(s_0)|=\limsup_{\eps\rightarrow 0}|w^\eps(R\nu,0)-\chi_{\nu}(s_0)|<\delta.
\]
Moreover the convergence is uniform in $(y_\eps,s_\eps)$. We can prove the following statement through a compactness argument.
\begin{lemma}\label{lem lem1}
{For any fixed $\delta>0$, there are positive constants $L_\delta,\eps_0$ and $R_0$ such that for all $(y,s)$ on the boundary of the domain of $w^\eps$, if $(y,s)\in S_{L_\delta}(y_0,s_0)$, we have for all $R\geq R_0$}
\begin{equation}\label{e0r0}
\sup_{\eps\leq\eps_0}|w^\eps(y+\eps^{\frac{1}{2}}R\nu,s)-\chi_\nu(s_0)|\leq \delta.
\end{equation}
\end{lemma}
\begin{proof}
Denote $\mu^{\eps}_{y,s}(x,t):=w^{\eps}(y+\eps^{\frac{1}{2}}x,s+\eps t)$. 
Suppose the claim fails, we can assume that there exists $\delta>0$ such that for some $R>0$ which can be arbitrarily large, there exist sequences $\eps_n\rightarrow 0$ and $(y_n,s_n)\rightarrow(y_0,s_0)$ (as $n\to\infty$) with $(y_n,s_n)$ is on the boundary of the domain of $w^{\eps_n}$, such that we have
\begin{equation*}
\limsup_{n\rightarrow \infty}|\mu^{\eps_n}_{y_n,s_n}(R\nu,0)-\chi_\nu(s_0)|> \delta.
\end{equation*}
Because $(y_n,s_n)$ is on the boundary, $\eps^{-\frac{1}{2}}y_n\cdot\nu=o(|s_n|)=o(1)$.
By Lemma \ref{contu}, by passing to a subsequence of $n\to\infty$, $\mu^{\eps_n}_{y_n,s_n}$ converges locally uniformly to $\mu^{z_0,\tau_0}$ which solves \eqref{I} with $z=z_0+s_0\nu,\tau=\tau_0$ or (R1) with $z=z_0,cs=s_0,\tau=\tau_0$. By Lemma \ref{bd lyer lmt R1}, the boundary layer limit $\lim_{R\to\infty}\mu^{z_0,\tau_0}(R\nu,0)\to \chi_\nu(s_0)$ as $R\to\infty$.
Therefore 
\begin{align*}
    &\limsup_{n\rightarrow \infty}|\mu^{\eps_n}_{y_n,s_n}(R\nu,0)-\chi_\nu(s_0)|\\
    \leq\,& \limsup_{n\rightarrow \infty}|\mu^{\eps_n}_{y_n,s_n}(R\nu,0)-\mu^{z_0,\tau_0}(R\nu,0)|+|\mu^{z_0,\tau_0}(R\nu,0)-\chi_\nu(s_0)|\\
    <\,& \delta
\end{align*}
if we take $R$ large enough, which leads to the contradiction.
% By continuity of $\chi_\nu(\cdot)$, we proved the claim. 
\end{proof}

Now we turn to the proof of Proposition \ref{contpsi}.
More generally, for any $(y_\eps,s_\eps) \rightarrow (y_0,s_0) $ such that $(x_\eps+\eps^\frac{1}{2}y_\eps+c_\eps \eps t,t_\eps+\eps s_\eps)\in\overline{\Omega}_T$, write $y_\eps=y_\eps'+r_\eps$ so that $(y_\eps',s_\eps)$ is on the boundary and $r_\eps$ is to $\nu$ direction. Then obviously $r_\eps\rightarrow 0$. The goal is to show that for any $\delta\in(0,1)$ and $R$ large enough, we have
\begin{equation}\label{wantos}
\limsup_{\eps\rightarrow 0}|w^\eps(y_\eps+\eps^\frac{1}{2}R\nu,s_\eps)-\chi_{\nu}(s_0)|\leq 3\delta.
\end{equation}

Let $\Omega':=(\Omega^\eps+\eps^\frac{1}{2}R\nu)\cap S_{L_\delta}(y_0,s_0)$ where $S_{L_\delta}(y_0,s_0)$ is the parabolic cylindar and 
\begin{equation*}
\Omega^\eps:=\bigcup_t \left(\eps^{-\frac{1}{2}}(\Omega(t_\eps+\eps t)-x_\eps-c_\eps\eps t \nu)\times\{t\}\right).
\end{equation*} 
According to \eqref{e0r0}, if $L_\delta$ is small enough there exist $\eps_0,R_0$ such that for all $R\geq R_0$
\[\sup_{\eps\leq \eps_0}|w^{\eps}(y,s)-\chi_\nu(s_0)|<\delta\]
for all $(y,s)\in (\partial_l\Omega^\eps+\eps^\frac{1}{2}R\nu)\cap S_{L_\delta}(y_0,s_0)=:\partial_l\Omega'$.

Now we want to apply Corollary \ref{cmp coro} for $w^\eps-\chi_\nu(s_0)$. Since $w^\eps$ solves \eqref{derive R2} and so does $ w^\eps-\chi_\nu(s_0)$, the equation meets the condition of the corollary. By the scaling, $\Omega'$ satisfies the condition of the domain. By the corollary we obtain that there exists $r_\delta<L_\delta$ such that for all $\eps$ small enough and all $y,s\in \Omega^\eps\cap S_{r_\delta}(y_0,s_0)$ we have
\[|w^\eps(y+\eps^\frac{1}{2}R\nu,s)-\chi_{\nu}(s_0)|\leq 3\delta\]
which implies \eqref{wantos}. Then $w^\eps$ converges in the half-relaxed limit sense on the boundary which finishes the proof.
\end{proof}

With the help of the above Lemma \ref{contu} and Proposition \ref{contpsi}, we are going to show the main theorem in this section. %Lemma \ref{contpsi} will be applied to prove claim \eqref{claimunif} which is about some local uniform convergence result in Case 2. For Case 1, the proof will be similar by applying Lemma \ref{contu} instead. 

\begin{proof}(of Theorem \ref{homol})
Let us only prove \eqref{irr bndry cnvrg}, while the proof for \eqref{r bndry cnvrg} is similar to the one of Lemma 5.2 \cite{wf} by applying Lemma \ref{contu}. As done in Lemma \ref{lem lem1}, first we show the following statement:
\textit{for any fixed $\delta>0$, we can find positive constants $L_\delta,\eps_0$ and $R_0$ such that for any $(x,t)\in \partial_l\Omega\cap S_{L_\delta}{(x_0,t_0)}$, there holds for $R\geq R_0$}
\begin{equation}\label{claimunif}\sup_{\eps\leq\eps_0}|u^{\eps}(x+\eps^\frac{1}{2}R\nu,t)-\bar{g}(x_0,t_0)|\leq \delta.
\end{equation}
The proof is similar to the one of Lemma \ref{lem lem1}.
Suppose the statement fails, there exists $\delta>0$ such that for any fixed $R$ which can be arbitrarily large, there are sequences $\eps_n\rightarrow 0$ and $(x_n,t_n)\in\partial_l\Omega$ satisfying $(x_n,t_n)\rightarrow (x_0,t_0)$ such that,
\begin{equation}\label{ineq1}
|u^{\eps_n}(x_n+\eps^\frac{1}{2}R\nu,t_n)-\bar{g}(x_0,t_0)|>\delta.
\end{equation}
Let
\[w^{\eps_n}_{n}(x,t)=u^\eps(x_n+\eps^{\frac{1}{2}}x+c(t_n)\nu \eps t, t_n+\eps t)\]
and then $|w^{\eps_n}_{n}(\eps^\frac{1}{2}R\nu,0)-\bar{g}(x_0,t_0)|>\delta.$
However by Proposition \ref{contpsi}, by passing to a subsequence of $n\to \infty$, $w^{\eps_n}_{n}$ converges locally uniformly to $w^z$ for some $z$ which solves (R2). By Lemma \ref{R1R2} the boundary layer limit of $w^z$ (independent of $z$) equals $\psi=\bar{g}(x_0,t_0)$. Therefore
\begin{align*}
    &\limsup_{n\to\infty}|w^{\eps_n}_{n}(\eps^\frac{1}{2}R\nu,0)-\bar{g}(x_0,t_0)|\\
    \leq\,& \limsup_{n\to\infty}|w^{\eps_n}_{n}(\eps^\frac{1}{2}R\nu,0)-w^z(\eps^\frac{1}{2}R\nu,0)|+|w^{z}(\eps^\frac{1}{2}R\nu,0)-\bar{g}(x_0,t_0)|<\delta
\end{align*}
for some $R(\delta)$ large enough which contradicts with \eqref{ineq1} and we proved the claim.

\smallskip

As done in Lemma \ref{lem lem1}, we are going to apply Corollary \ref{cmp coro}. We use
\[\tilde{w}^{\eps}(x,t)=u^\eps(x_0+\eps^{\frac{1}{2}}R\nu+\eps^{\frac{1}{2}}x+c_0\nu \eps t, t_0+\eps t)-\bar{g}(x_0,t_0)\]
and by the claim
$|\tilde{w}^\eps|\leq \delta$
for all $\eps\leq \eps_0$ on the boundary of its domain near the origin: $\partial_l\tilde{\Omega}^\eps\cap S_{\eps^{-\frac{1}{2}}L_\delta}$, where
\[\tilde{\Omega}^\eps:=\bigcup_t \left(\eps^{-\frac{1}{2}}(\Omega(t_0+\eps t)-x_0-c_0\eps t \nu)\times\{t\}\right).\]
Set the domain to be $D=\{(x,t)\in\partial_l\tilde{\Omega}^\eps+r\nu\,|\, r\in[0,\eps^{-\frac{1}{2}}L_\delta]\}.$ Denote $\|\partial_l\Omega_T\|_{\mathcal{C}^1}$ as the $\mathcal{C}^1$ norm of the lateral boundary of $\Omega_T$ and then 
\[\|\partial_l\tilde{\Omega}^\eps\|_{\mathcal{C}^1}=\eps^{\frac{1}{2}}\|\partial_l\Omega_T\|_{\mathcal{C}^1}.\]
Let $c>0$ be the small constant from Corollary \ref{cmp coro}. It is straightforward that 
\[(\eps^{-\frac{1}{2}}L_\delta)(\eps^\frac{1}{2} \|\partial_l\Omega_T\|_{\mathcal{C}^1}+\eps^{\frac{1}{2}}\|c\|_\infty)\leq c,\quad (\eps^{-\frac{1}{2}}L_\delta)^2\eps\leq c\delta\]
when $L_\delta$ is small only depending on $\delta$ and universal constants. Thus the corollary yields that there exists $r_\delta>0$ such that $|\tilde{w}^\eps|\leq 3\delta$ in $S_{\eps^{-\frac{1}{2}}r_\delta}$, which implies \eqref{irr bndry cnvrg}.

%$\lim_{R\rightarrow\infty}\psi^z_n(R\nu,0)$ exists and the limit $\bar{g}(x_n,t_n)$ is independent of $z$. Also Theorem \ref{tcont} shows that $|\bar{g}(x_n,t_n)-\bar{g}(x_0,t_0)|<\frac{\delta}{2}$ for $n$ large enough. We can assume that there exists another sequence $\eps_n'\rightarrow 0$ and $R_1$ large enough, that for $n$ large enough
%\begin{equation}\label{ineq2}
%|\psi^{\eps_n'}_{n}(R_1\nu,0)-\bar{g}(x_0,t_0)|<\delta.
%\end{equation}
%Notice that these functions $\{\psi^{\eps_n'}_{n},\psi^{\eps_n}_{n}\}$ satisfy the conditions in Lemma \ref{contpsi} along subsequences. We find that both $\psi^{\eps_n'}_{n},\psi^{\eps_n}_{n}$ converge locally uniformly to the same function $\psi^z_0$ which contradicts with inequalities \eqref{ineq1} \eqref{ineq2}. We proved the claim.

\end{proof}

%\begin{remark}If the lateral boundary $\partial_l\Omega$ contains the two cases $\Gamma_1,\Gamma_2$, in examples in subsection \ref{discont}, we show that the $\bar{g}$ we defined can be discontinuous. From the above discussion, in fact we can argue that the discontinuity of homogenized boundary data is unavoidable.\end{remark}

\section{Homogenization on the Bottom Boundary}\label{bot}

In this section we briefly discuss the homogenization of $u^\eps$ on the bottom boundary. Since the bottom boundary is just flat, the proofs are actually simpler than those in previous sections. Below we start with a localized comparison lemma on $\{(x,t)\,|\, t\geq 0\}$.
\begin{lemma}\label{lemmab}
Let $C_0>c_0>0$ and $L>R>0$. Assume $w$ satisfies the following equation
\[
\left\{\begin{aligned}
&\frac{\partial}{\partial t}w-\mathcal{P}^+(D^2 w)\leq 0& \text{ in }&\quad  Q^0_L,\\
&w(x,0)\leq c_0&\text{ on }&\quad  \{t=0\},\\
&w(x,t)\leq C_0& \text{ in }&\quad  \overline{Q^0_L},
\end{aligned}
\right.
\]
where $Q^0_L:=\left\{(x,t)\,|\,|x|< L,\,0<t<L^2\right\}$. Then there exists a constant $C$ only depending on $C_0,d,\Lambda$ such that
\[w(x,t)\leq c_0+C{R^2}/L^{2} \quad \text{ for }(x,t)\in \overline{Q}^0_R.\]

\end{lemma}
\begin{proof}
Let $C_1=C_0\max\left\{2\Lambda d,1\right\}$.
Consider the following barrier 
\[\phi(x,t)=C_0 L^{-2}|x|^2+c_0+C_1L^{-2}t.\]
It is direct to check that 
\[\frac{\partial}{\partial t}\phi- \mathcal{P}^+(D^2\phi)\geq 
C_1L^{-2}-C_0L^{-2}\Lambda 2d\geq 0, \]
\[
 \phi(x,t)\geq C_0 \quad\text{ when }|x|=L \text{ or }t=L^2\]
 and $\phi(x,0)\geq c_0$ for all $x$.
 By definition of viscosity solution, $w(x,t)\leq \phi(x,t)$. Then restricting $(x,t)$ to $\overline{Q}^0_R$ finishes the proof.
\end{proof}

%\begin{corollary}
%Suppose $v_1(x,t),v_2(x,t)$ are two bounded solutions to equation \eqref{eqns1} in $\mathbb{R}^d\times \mathbb{R}^+$. Then
%\[\sup_{t\geq 0,x\in \mathbb{R}^d}\left(v_1(x,t)-v_2(x,t)\right)_+\leq \sup_{x\in \mathbb{R}^d}\left(v_1(x,0)-v_2(x,0)\right)_+.\]\end{corollary}

As before, we will formally derive the cell problem, from which we define the homogenized boundary data as a boundary layer limit. After showing the continuity property, we prove bottom boundary homogenization. 

Fix a point $x_0\in\Omega(0)$ and set
\[v^\eps(x,t):=u^\eps(x_0+\eps x,\eps^2 t).\]
Write $\Omega^\eps:=\bigcup_t(\Omega^\eps(t)\times\{t\})$ with $\Omega^\eps(t)=\eps^{-1}(\Omega(\eps^2 t)-x_0)$, and then $v^\eps$ satisfies
\begin{equation*}
\left\{\begin{aligned}&
\frac{\partial}{\partial t}v^\eps(x,t)-\eps^2 F(\eps^{-2}D^2v^\eps,x_0+\eps{x},\eps^{-1}x_0+x,\eps^2 t,t)=0 & \text{ in }& \quad  \Omega^\eps,\\
&v^\eps(x,0)=g(x_0+{\eps}x,\eps^{-1}x_0+x,0,0)&  \text{ on }&\quad  \Omega^\eps(0).
\end{aligned}
\right.
\end{equation*}
Suppose $\eps^{-1}x_0$ converges to $z$ as $\eps\rightarrow 0$ in $\mathbb{R}^d/\mathbb{Z}^d$ along subsequences. Formally, we get the cell problem
\begin{equation*}
(B)\quad\quad\left\{\begin{aligned}&
\frac{\partial}{\partial t}v^z(x,t)- F^{0}(D^2v^z,x+z,t)=0 & \text{ in }&\quad   \mathbb{R}^d\times\mathbb{R}^+,\\
&v^z(x,0)=g^0(x+z,0)&  \text{ on }&\quad  \mathbb{R}^d.
\end{aligned}
\right.
\end{equation*}
By Lemma \ref{lemmab} and Perron's method, there exists a unique solution to (B).

\smallskip

%If we take upper and lower half-relaxed limits of $v^\eps$ along the subsequence denoted as $v^{\ast z}$ and $v_{\ast}^z$ respectively, with the help of the stability of viscosity solutions, we obtain $v^{\ast z}$ is a subsolution to equation (B) while $v_{\ast}^z$ is a supersolution.By comparison principle, we get $v^{\ast z}= v_{\ast }^z=v^z$. %The following lemma give from limits of which we will derive the homogenized boundary bottom data.

In the following lemma, we study the boundary layer limit of (B).

\begin{lemma}\label{bot bdry lyer limit}
Let $v^z$ be the bounded solution to equation (B). Then $ v^{z}(x,R)$ converges uniformly for all $x,z$ as ${R\rightarrow \infty}$ and the boundary layer limit $\varphi_b$ is independent of $x,z$. Furthermore, if $F^0(M,y,s)=F^0(M)$ and is linear in $M$, then $\varphi_b$ satisfies
\[\varphi_b=\int_{[0,1)^{n}}g^0(y,0)dy.\]
\end{lemma}
\begin{proof}
By the periodicity of $F^0(M,y,s),g^0(y,0)$ in $y,s$ variables, we have
$v^{z}(x+\mathbb{Z},t)=v^{ z}(x,t).$
So we only need to consider $v^{ z}(x,t)$ when $|x|\leq \sqrt{d}$. Since $v^z$ is bounded, Corollary \ref{cor LW} implies that $|v^{ z}(x,R)-v^{ z}(0,R)|\leq CR^{-\beta} $ for some universal positive constants $C,\beta$. By Lemma \ref{lemmab}, for $t>0$,
\[\osc_{(\{(x,t)\in\mathbb{R}\times[R,\infty) \})}\left(v^{ z}(\cdot,\cdot)\right)\leq \osc_{(\{x\in\mathbb{R}\})}\left(v^{ z}(\cdot,R)\right).\] 
This shows the existence of the boundary layer limit $\varphi_b^z:=\lim_{R\to\infty}v^z(x,R)$ and
\begin{equation}
    \label{rate bt bdry}|v^z(x,R)-\varphi_b^z|\leq CR^{-\beta}.
\end{equation}

Next for a different $z'$, take $v^{ z'}$ to be the equation (B) with $z$ replaced by $z'$.
Take $\tilde{v}(x,t)=v^{ z'}(x+z-z',t)$ which then solves the same equation as $v^{ z}(x,t)$ does. By uniqueness, $v^{ z'}(x+z-z',t)=v^{ z}(x+z,t)$. Since the boundary layer limit is independent of $x$ and thus $\varphi_b^z=\varphi_b^{z'}=:\varphi_b(F^0,g^0)$.

For the second claim, the proof is similar to the one in Lemma 3.6 \cite{wf} and Lemma 7.1 \cite{feldman2018continuity} where elliptic equations and systems are concerned. Indeed we can consider a linear map $T$ which maps periodic functions $g^0(\cdot)$ to $\varphi_b$. It can be checked that $T$ is continuous, translation invariant and $T(1)=1$. By Riesz Representation theorem $T(g^0(\cdot))=\int_{[0,1)^{n}}g^0(y)dy$.
\end{proof} 

%We define the limit in the lemma as the homogenized boundary data at $(x_0,0)$. \[\bar{g}(x_0,0)=\bar{g}(F^0,g^0).\]

\begin{definition}\label{bdef}
For any {$x_0\in \Omega(0)$}, define \[\bar{g}(x_0,0)=\varphi_b(F^{x_0,0},g^{x_0,0})\]
where $\varphi_b$ is the boundary layer limit of $v^z$ given in the previous lemma.
\end{definition}

Now we prove the main theorem of this section.

\begin{theorem}(Bottom boundary homogenization)\label{homob}
\begin{enumerate}
    \item[(i.) ] 
$\bar{g}(\cdot,0)$ is continuous on $\Omega(0)$. 

\item[(ii.) ]For any $ \delta>0$ and {$x_0\in\Omega(0)$}, there exists $R_0>0$ such that the following holds. Take any sequence  $(x_\eps,t_\eps)\in \overline{\Omega}_T$ such that $(x_\eps,t_\eps)\rightarrow (x_0,0)$ as $\eps\to 0$, we have for all $R>R_0$
\begin{equation}\label{bothomo}
\limsup_{\eps\rightarrow 0}|u^\eps(x_\eps,t_\eps+\eps R)-\bar{g}(x_0,0)|\leq \delta.
\end{equation}

\end{enumerate}
\end{theorem}
\begin{proof}
The continuity property follows from the fact that $F^{x,t},g^{x,0}$ are uniformly continuous in $x$ and the convergence of the boundary layer limit is uniform in $x$ (see \eqref{rate bt bdry}). 

For the second part, let $v^\eps(x,t)=u(x_\eps+x,t_\eps+t)$.
Suppose along a subsequence of $\eps\to 0$, $\frac{x_\eps}{\eps}\to z$ in $\mathbb{R}^d/\mathbb{Z}^d$.
We take upper and lower half-relaxed limits of $v^\eps$ along the subsequence and we denote the limits as $v^{\ast z}$ and $v_{\ast}^z$ respectively. As before, $v^{\ast z}$ is a subsolution to equation (B) while $v_{\ast}^z$ is a supersolution.

By continuity properties, $v^{\ast z}=v_{\ast}^z=\bar{g}(x_0,0)$. Hence by comparison principle $v^{\ast z}\leq v_{\ast}^z$. Also since the reverse inequality holds by definition, we have $v^{\ast z}= v_{\ast}^z=v^z$. Note for all $z,x$, $v^z(x,R)\to \bar{g}(x_0,0)$ uniformly. Therefore we proved \eqref{bothomo}.

\end{proof}

\section{Uniqueness and Conclusions}\label{lastsection}

In the previous sections we identified the homogenized boundary data on $\Gamma_T$ which is most of the lateral boundary and the whole bottom boundary.
In this section, we want to prove the comparison principle for fully nonlinear parabolic equations even if the ordering on the lateral boundary only holds outside a small subset. Once this is done and suppose $\Gamma_T$ almost covers the whole boundary, we can show the homogenization of \eqref{main eqn}.

%Moreover, this will lead to the desired uniqueness result of the homogenized problem we have. 
To measure the subsets of the boundary, let us introduce the following parabolic Hausdorff dimension of subsets in $\mathbb{R}^{d+1}$. 

\begin{definition}\label{prblc HD}
Suppose $\Sigma$ is a subset of $\mathbb{R}^d\times \mathbb{R}$. Define the $\mathfrak{d}$-dimensional parabolic Hausdorff content:
\[\mathcal{C}^\mathfrak{d}_{\mathcal{P}}(\Sigma)=\inf\left\{\sum_j r_j^\mathfrak{d}\,|\, \Sigma\subset \bigcup_j S_{r_j}(x_j,t_j)\right\}.\]
Here $S_r(x,t)$ is the parabolic cylinder given by \eqref{prblc}. We say that $\Sigma$ has parabolic Hausdorff dimension $d_\mathcal{P}$ if
\[d_\mathcal{P}(\Sigma)=\inf\left\{\mathfrak{d}\geq0 \,| \,\mathcal{C}^\mathfrak{d}_{\mathcal{P}}(\Sigma)=0  \right\}.\]
\end{definition}
We remark here that if denoting the standard Hausdorff dimension by $d_\mathcal{H}(\cdot)$, then we have for all $\Sigma\subset\mathbb{R}^{d+1}$\[{d_\mathcal{H}}(\Sigma)\leq d_\mathcal{P}(\Sigma)\leq 2d_\mathcal{H}(\Sigma).\] 

%Let $G(M,x,t)$ be a function on $\mathcal{M}\times\mathbb{R}^{d+1}$. We state the comparison principle below.
\begin{theorem}\label{uniq}
Suppose $G$ is a function on $\mathcal{M}\times\mathbb{R}^{d+1}$ satisfying (G) and $\Omega_T\subset \mathbb{R}^{d+1}$ is a space-time parabolic domain satisfying (O). Let $\Sigma \subset \partial_l\Omega_T$ be such that
\[\mathcal{C}^{d_0}_\mathcal{P}(\Sigma)=0 \quad \text{ where } d_0:={d\lambda}/{\Lambda}.\]
(In particular, this condition holds when the Hausdorff dimension of $\Sigma$ is less than $\frac{d\lambda}{2\Lambda}$.)
If two bounded functions $u,v$ are respectively upper and lower semi-continuous in $\Omega_T$ and they satisfy
\[
\left\{\begin{aligned}&
\frac{\partial}{\partial t}v(x,t)-G(D^2v,x,t)\leq \frac{\partial}{\partial t}u(x,t)-G(D^2u,x,t)&  \text{ in }&\quad \Omega_T,\\
    & \limsup_{(y,s)\rightarrow (x,t)} v(y,s)\leq \liminf_{(y,s)\rightarrow (x,t)}u(y,s)& \text{ on }&\quad \partial_p\Omega_T \backslash \Sigma,
\end{aligned}
\right.
\]
then $v\leq u$ in $\Omega$.
\end{theorem}
In order to prove the theorem, we construct a singular super solution to the fully nonlinear parabolic equation with Pucci's operator.
\begin{lemma}\label{Phi}
Let $\mathcal{P}^+$ be the Pucci's extremal operator with parameters $\Lambda\geq \lambda>0$. Set \[ \Phi(x,t)=\left\{
\begin{aligned}&
t^{-\frac{d\lambda}{2\Lambda}}e^{-\frac{|x|^2}{4\Lambda t}} &   \text{ if  }\quad& t>0,\\
&0 &  \text{ if  }\quad& t\leq0.
\end{aligned}\right.\]
Then we have $\Phi(kx,k^2t)=k^{-\frac{d\lambda}{\Lambda}}\Phi(x,t)$ and $\Phi$ is a singular super solution to
\[\frac{\partial \Phi}{\partial{t}}-\mathcal{P}^+(D^2\Phi)\geq 0 \quad \text{ in }\quad\mathbb{R}^{d+1}\backslash (0,0).\]
\end{lemma}

The proof follows from a direct computation. As a remark, Armstrong, Sirakov and Smart \cite{singular} constructed singular solutions to general fully nonlinear elliptic equations and the parallel result for the parabolic equations remains to be studied. For us, one singular solution is enough for the purpose.

\begin{proof}(of Theorem \ref{uniq})
Set $w=v-u$ and it satisfies
\[
\left\{\begin{aligned}&
\frac{\partial}{\partial t}w-\mathcal{P}^+(D^2 w)\leq 0 &\text{ in }&\quad \Omega_T,\\
& \limsup_{(y,s)\rightarrow (x,t)} w(y,s)\leq 0 & \text{ on }&\quad \partial_p\Omega_T \backslash \Sigma.
\end{aligned}
\right.
\]
Let $m_1=\min\{\Phi(x,1)\,|\, |x|\leq 1 \}$, $m_2=\max\{\Phi(x,t)\,|\,t=1 \text{ or }|x|=1\}$ and it is obvious that $m_2>m_1>0$. Since $u,v$ are bounded, we assume $|w|\leq C_0$. For any fixed $\delta>0$, by the $0$ dimension assumption, we can assume that $\Sigma$ is covered a set of countably many parabolic cylinders $\{S_{r_j}(x_j,t_j)\}_j$ and 
\[\sum_{j\geq 1} r_j^{d_0}<\delta^2 \frac{m_1}{C_0 m_2 2^{d_0}}=:\frac{\delta^2}{M m_2}.\]
Let $\Phi$ be as in Lemma \ref{Phi} and consider the following barrier
\[\Phi_\delta(x,t)=M \sum_j r_j^{d_0}\Phi(x-x_j,t-t_j+2r_j^2)=:M \sum_j r_j^{d_0}\Phi_j(x,t).\]
By Lemma \ref{Phi}, this $\Phi_\delta$ is a super solution to
\[{\partial_t}\Phi_\delta-\mathcal{P}^+(D^2 \Phi_\delta)\geq M\sum_j  r_j^{d_0}(\frac{\partial}{\partial t}\Phi_j-\mathcal{P}^+(D^2 \Phi_j) )\geq 0 \quad\text{ in }\Omega_T.\]

Note for any $(x,t)\in S_{r_j}(x_j,t_j)$ for some $j$, we have
\[3r_j^2\geq t-t_j+2r_j^2\geq r_j^2,\quad \left|(x-x_j)(t-t_j+2r_j^2)^{-\frac{1}{2}}\right|\leq 1.\] 
Therefore
\[\Phi_j(x,t)\geq (t-t_j+2r_j^2)^{-\frac{d_0}{2}}\Phi\left((x-x_j)(t-t_j+2r_j^2)^{-\frac{1}{2}},1\right) \geq (2r_j)^{-d_0}m_1.\]
So we have
\[\Phi_\delta(x,t)\geq 2^{-d_0}Mm_1= C_0\geq w(x,t).\]
For points on the boundary outside $\Sigma$, we have $\Phi_\delta(x,t)\geq 0\geq w(x,t)$. Then by comparison principle, $w\leq \Phi_\delta$ in $\Omega_T$.

Now for any $(x,t)\in\Omega_T$ that $(x,t)\notin \bigcup_j S_{\delta^{{1}/{d_0}}}(x_j,t_j)$, we have either
\[|x-x_j|\geq \delta^{\frac{1}{d_0}}\text{ or } |t-t_j|\geq \delta^{\frac{2}{d_0}},\]
and then
$\Phi_j(x,t)\leq \delta^{-1}m_2$. This gives 
\[\Phi_\delta(x,t)\leq M\delta^{-1}m_2\sum_j r_j^{d_0}\leq \delta.\]
Since $w\leq \Phi_\delta$ in $\Omega_T$, letting $\delta\rightarrow 0$ shows $w\leq 0$ in $\Omega_T$, which finishes the proof.
\end{proof}

%\begin{remark}In particular the dimension assumption holds if the parabolic Hausdorff dimension of $\partial_l\Omega\backslash \Sigma$ is less than $d_0$ or even the Hausdorff dimension is less than $\frac{d_0}{2}$.\end{remark}

In all, let us put together the results obtained and conclude with the following main theorem of the paper.

\begin{theorem}\label{final}
%Let $\Omega_T$ be a time-dependent domain in $\mathbb{R}^{d+1}$ satisfying condition (O) and $F, g$ satisfies (F1)--(F4). Let $\Gamma_T=\Gamma_1(T)\cup\Gamma_2(T)\subseteq \partial_l\Omega_T$ be the union of the two Cases and write $d_0=\frac{\lambda d}{\Lambda}$. Suppose
Assume conditions (O)(F1)--(F4) hold. For any $T>0$, recall \eqref{bdry clss} and let $\Gamma_T=\Gamma_1(T)\cup\Gamma_2(T)\subseteq \partial_l\Omega_T$. Denote $d_0=\frac{\lambda d}{\Lambda}$ and suppose
\[\mathcal{C}^{d_0}_\mathcal{P}\left(\partial_l\Omega_T\backslash \Gamma_T\right)=0.\]
 Then for $t\in (0,T)$, the solutions $u^\eps(x,t)$ to \eqref{eqns1} converge locally uniformly to the unique solution $\bar{u}(x,t)$ of 
\begin{equation}\label{7.3}
\left\{\begin{aligned}&
\frac{\partial}{\partial t}\bar{u}(x,t)-\bar{F}(D^2\bar{u},x,t)=0,& \text{ in }&\quad \Omega_T\\
&\bar{u}(x,t)=\bar{g}(x,t)& \text{ on }&\quad  \Gamma_T,\\
&\bar{u}(x,0)=\bar{g}(x,0)&\text{ on }&\quad   \overline{\Omega(0)}
\end{aligned}
\right.
\end{equation}
where $\bar{F}$ is the homogenized operator associated with $F$ given by Theorem \ref{homo} and $\bar{g}$ is a continuous on {$\Gamma_T\cup(\Omega(0)\times\{0\})$} given by \eqref{def bar g} and \eqref{bdef}. 
\end{theorem}
\begin{proof}
%In order to use comparison principle, we first d
%Define for $x\in \partial \Omega(0)$, $\bar{u}(x,0)=\lim_{\substack{y\rightarrow x,\\y\in \Omega(0)}}\bar{g}(y,0)$. This definition makes sense because of the continuity of $\bar{g}(y,0)$.

%Let $u^\eps$ be solutions to equations \eqref{eqns1}.

Let $u^*$ and $u_*$ be respectively the upper and lower half-relaxed limits of $u^\eps$. Then by Theorem \ref{homo}, $u^*$, $u_*$ are respectively sub and super solutions to equation \eqref{7.3}. By Theorem \ref{homol} and Theorem \ref{homob}, we know that for $(x,t)\in \Gamma_T\cup{(\Omega(0)\times\{0\})}$
%\[\liminf_{\Omega\ni(x',t')\rightarrow (x,t)}u^\eps(x',t')=\limsup_{\Omega\ni(x',t')\rightarrow (x,t)}u^\eps(x',t')=\bar{g}(x,t).\]
\[u^*(x,t)=u_*(x,t)=\bar{g}(x,t).\]
The continuity of $\bar{g}$ is proved in Theorem \ref{tcont} and Theorem \ref{homob}.
By the assumption, the $d_0$-dimensional parabolic Hausdorff content of the set $\partial_l\Omega_T\backslash\Gamma_T$ equals $0$. We apply Theorem \ref{uniq} and find out
\[u_*(x,t)\geq u^*(x,t) \quad\text{ in }\Omega_T.\]
Since the other direction of the above inequality holds trivially by definitions, we have $u_*= u^*=\bar{u}$. This shows that $u^\eps$ converges locally uniformly in $\Omega_T$ and the limit equals $\bar{u}$, the unique solution to \eqref{7.3}. Uniqueness of solutions of \eqref{7.3} again follows from Theorem \ref{uniq}. %So $u^\eps$ converges locally uniformly to $\bar{u}$ in $\Omega_T$.
\end{proof}

Moreover if the associated homogenized operator $\bar{F}$ is linear, we have the following theorem.

\begin{theorem}\label{linearversion}
Assume conditions (F1)--(F4)(O) hold, the homogenized operator $\bar{F}(M,x,t)$ is linear in $M$ for all $x,t$ and we have $\partial_l\Omega\subset\overline{\Gamma}_T$. Then the homogenized boundary data $\bar{g}$ is continuous on ${\partial_p\Omega_T}$. The solutions $u^\eps(x,t) $ to \eqref{eqns1} converge locally uniformly to the unique solution $\bar{u}(x,t)$ to \eqref{homosolution}.
%\begin{equation*}
%\left\{\begin{aligned}&
%\frac{\partial}{\partial t}\bar{u}(x,t)-\bar{F}(D^2\bar{u},x,t)=0&  \text{ in }&\quad \Omega_T\\
%&\bar{u}(x,t)=\bar{g}(x,t)& \text{ on }&\quad   \partial_p\Omega_T.
%\end{aligned}
%\right.
%\end{equation*}
Furthermore suppose that in a neighbourhood of $(x,t)\in\partial_p\Omega_T$ the original operator $F(M,x,y,t,s)$ is independent of $y,s$ and is linear in $M$, then we have
\begin{align*}
&\bar{g}(x,t)=\int_{[0,1)^{n+1}}g(x,y,t,s)dyds &\quad\text{ if }t>0,\\
&\bar{g}(x,0)=\int_{[0,1)^{n}}g(x,y,0,0)dy&\quad\text{ otherwise}.
\end{align*}

\end{theorem}
\begin{proof}

Continuity extension of the homogenized boundary data $\bar{g}$ on ${\partial_p\Omega_T}$ follows from Remark \ref{rmrk linear}. Having a continuous boundary data $\bar{g}$ defined on the whole parabolic boundary, the proof of the convergence follows from the proof of Theorem \ref{final}.

For the second statement, first suppose $(x,t)\in \Gamma_1$. From the assumption, we can write 
\[F(M,x,y,t,s)=F(M,x,t).\] 
The operator in the cell problem (I) is simply $F^{x,t}(M)=F(M,x,t)$. By Lemma \ref{concave}, the boundary layer limit of (I) equals the linear average of $g^{x,t}(y,s)=g(x,y,t,s)$ about $y,s$ variables. If $x\in \Omega(0)$, we apply Lemma \ref{bot bdry lyer limit} instead to get the same result. In the case when $(x,t)\in\Gamma_2$, due to Lemma \ref{bd lyer lmt R1},  $\chi_{\nu}(\cdot )$ is periodic with periodicity $|\hat{\nu}|^{-1}$ where $\hat{\nu}$ is such that $\hat{\nu}\in\mathbb{Z}^d$ is irreducible and $\nu=\frac{\hat{\nu}}{|\hat{\nu}|}$. Let us write one smallest periodic block in $\partial P_\nu$ as $\mathbb{T}_\nu$. Then using linearity of the operators and applying Lemma \ref{concave} twice, we get
\begin{align*}
\bar{g}(x,t)&=|c\hat{\nu}|\int_{0}^{|c\hat{\nu}|^{-1}} \chi_\nu(c \tau)d\tau\\
&=\frac{1}{\text{vol}\left(\mathbb{T}_\nu\times [0,{|\hat{\nu}|^{-1}})\right)}\int_{[0,|\hat{\nu}|^{-1})}\int_{\mathbb{T}_\nu\times [0,1)} g^{x,t}(z+\tau\nu,s )dzdsd\tau\\
&=\int_{[0,1)^n\times [0,1)} g^{x,t}(y,s)dyds\\
&=\int_{[0,1)^{n+1}} g(x,y,t,s)dyds . 
\end{align*}
We finished the proof.
%Assume $F$ is linear and independent of $y,s$ only locally near some $(x,t)$, then we have the averaging property locally.
\end{proof}

%The author would like to thank Prof. Inwon Kim for suggesting this problem as well as all the stimulating guidance and the fruitful discussions. Also the author is grateful to Olga Turanova and William M. Feldman for reading this paper and providing helpful comments. 

%\end{acknowledgement}

%----------------------------------------------------------------------------------------
%	REFERENCE LIST
%----------------------------------------------------------------------------------------
%\bibliography{homo}{}
%\bibliographystyle{plain}

\end{document}